\documentclass{amsart}
\usepackage{amsmath}
\usepackage{amsfonts,exscale,amssymb,amsthm,epsfig}
\usepackage{upref}
\usepackage{graphicx}
\usepackage{mathtools}
\usepackage{enumerate}
\usepackage[usenames, dvipsnames]{color}
\usepackage{accents} 

\input xy
\xyoption{all}

\usepackage{ifthen}
\usepackage{tikz}
\usepackage{appendix}
\usepackage{verbatim}
\usetikzlibrary{decorations.pathmorphing}
\usetikzlibrary{calc}
\usepackage{hyperref}

\renewcommand{\comment}[1]{}
\newcommand{\eq}{\begin{equation}}
\newcommand{\en}{\end{equation}}
\newcommand{\rr}{\mathbb{R}}
\newcommand{\NN}{\mathbb{N}}

\newcommand{\mcal}[1]{\mathcal{#1}}

\newcommand{\combi}[2]{\begin{pmatrix}#1 \\ #2 \end{pmatrix}}

\newcommand{\nin}{\noindent}
\newcommand{\tbf}{\textbf}

\renewcommand{\int}{\mathrm{int}}
\newcommand{\parent}[1]{\accentset{\leftarrow}{#1}}

\newcommand{\bN}{\mathbb{N}}

\newcommand{\bT}{\mathbb{T}}
\newcommand{\bX}{\mathbb{X}}
\newcommand{\bY}{\mathbb{Y}}

\newcommand{\cC}{\mathcal{C}}

\newcommand{\cN}{\mathcal{N}}

\newcommand{\fr}{\mathbf{r}}
\newcommand{\fs}{\mathbf{s}}
\newcommand{\ft}{\mathbf{t}}

\newcommand{\fv}{\mathbf{v}}
\newcommand{\wi}{{\tilde{\textit{\i}}}}
\newcommand{\decotree}[2]{\bT_{[#1]}^{\bullet(#2)}}
\newcommand{\DirM}{\mathrm{DM}}
\newcommand{\tshape}[1]{\mathbb{T}^{}_{[#1]}}
\newcommand{\sparent}[1]{\accentset{\longleftarrow}{\{#1\}}}

\newcommand{\pP}{\mathbf{P}}
\newcommand{\pE}{\mathbf{E}}

\begin{document}

\theoremstyle{plain}
\newtheorem{thm}{Theorem}
\newtheorem{lemma}[thm]{Lemma}
\newtheorem{prop}[thm]{Proposition}
\newtheorem{cor}[thm]{Corollary}
\newtheorem{remark}[thm]{Remark}
\newtheorem{problem}{Problem}
\newtheorem{aside}{Aside}

\theoremstyle{definition}
\newtheorem{defn}{Definition}
\newtheorem{asmp}{Assumption}
\newtheorem{notn}{Notation}
\newtheorem{prb}{Problem}

\theoremstyle{remark}
\newtheorem{rmk}{Remark}
\newtheorem{exm}{Example}
\newtheorem{clm}{Claim}

\title[Consistent tree-valued down-up chains]{Projections of the Aldous chain on binary trees: Intertwining and consistency}


\address{\hspace{-0.42cm}Noah~Forman\\ Department of Statistics\\ University of Oxford\\ 24--29 St Giles'\\ Oxford OX1 3LB, UK\\ Email: noah.forman@gmail.com}             

\address{\hspace{-0.42cm}Soumik~Pal\\ Department of Mathematics \\ University of Washington\\ Seattle WA 98195, USA\\ {Email: soumikpal@gmail.com}}

\address{\hspace{-0.42cm}Douglas~Rizzolo\\ \ Department\ \ of\ \ Mathematical\ \ Sciences\\ \ University\ \ of\ \ Delaware\\ Newark DE 19716, USA\\ Email: drizzolo@udel.edu}

\address{\hspace{-0.42cm}Matthias~Winkel\\ Department of Statistics\\ University of Oxford\\ 24--29 St Giles'\\ Oxford OX1 3LB, UK\\ Email: winkel@stats.ox.ac.uk}           

\author{Noah Forman \and Soumik Pal \and Douglas Rizzolo \and Matthias Winkel}



\thanks{This research is partially supported by NSF grants DMS-1204840, DMS-1612483, and EPSRC grant EP/K029797/1}

\date{\today}

\begin{abstract}
Consider the Aldous Markov chain on the space of rooted binary trees with $n$ labeled leaves in which at each transition a uniform random leaf is deleted and reattached to a uniform random edge. Now, fix $1\le k < n$ and project the leaf mass onto the subtree spanned by the first $k$ leaves. This yields a binary tree with edge weights that we call a ``decorated $k$-tree with total mass $n$.'' We introduce label swapping dynamics for the Aldous chain so that, when it runs in stationarity, the decorated $k$-trees evolve as Markov chains themselves, and are projectively consistent over $k\le n$. The construction of projectively consistent chains is a crucial step in the construction of the Aldous diffusion on continuum trees by the present authors, which is the $n\rightarrow \infty$ continuum analogue of the Aldous chain and will be taken up elsewhere. Some of our results have been generalized to Ford's alpha model trees.  
\end{abstract}

\ \vspace{-24pt}

\maketitle

\section{Introduction}

The Aldous chain \cite{Aldous00,Schweinsberg02} is a Markov chain on the space of (rooted, for our purposes) binary trees with $n$ labeled leaves. Each transition of the chain, called a down-up move, has two stages. First, a uniform random leaf is deleted and its parent branch point is contracted away. Next, a uniform random edge is selected, we insert a new branch point into the middle of that edge, and we extend a new leaf-edge out from that branch point. This is illustrated in Figure \ref{fig:AC_move} where $n=6$ and the leaf labeled $3$ is deleted and re-inserted. The stationary distribution of this chain is the uniform distribution on rooted binary trees with $n$ labeled leaves.

\begin{figure}[t]\centering
\scalebox{0.8}{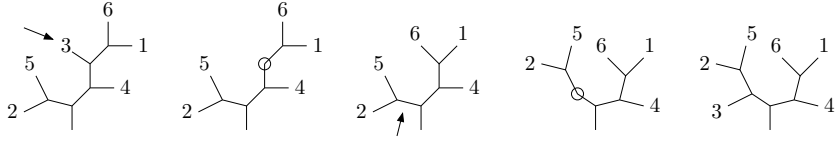}
 \caption{From left to right, one Aldous down-up move.\label{fig:AC_move}}
\end{figure}

Binary trees with leaf labels are called ``cladograms.'' Markov chains on cladograms are used in MCMC algorithms for phylogenetic inference. The Aldous chain is used in this role \cite{LVB}, alongside other chains such as the subtree prune and re-graft chain \cite{EW,mrbayes},
which makes larger changes to the tree structure at each step.

Despite its simplicity, it is difficult to understand how the structure of a tree evolves over time under the Aldous chain.  In his original work, Aldous gave bounds on the relaxation time for the chain.  Soon thereafter, Schweinsberg \cite{Schweinsberg02} showed that the relaxation time is $\Theta(n^2)$.   More recently, Pal \cite{Pal13} has studied connections, initially proposed by Aldous \cite{AldousDiffusionProblem}, between the Aldous chain and an extension of Wright--Fisher diffusions to negative mutation rates.  Specifically, suppose $(T_n(j))_{j\geq 0}$ is a stationary version of the Aldous chain on trees with $n$ leaves.  Let $L_1$ and $L_2$ be leaves sampled uniformly from $T_n(0)$ without replacement and let $L_0$ be the root of $T_n(0)$.  The subtree of $T_n(0)$ spanned by these leaves and the root consists of a root, a single branch point, and two leaves.  This branch point naturally partitions the tree $T_n(0)$ into three components. As the Aldous chain runs, leaves move among components until the branch point disappears, i.e.\ a component becomes empty.  If we let $m_i(j)$, $i\in \{0,1,2\}$, be the proportion of leaves of $T_n(j)$ in the component that initially contained $L_i$ (stopped when the branch point disappears), then
\begin{equation}\label{eq:negativewf}
 \left( m_0^{(n)}(\lfloor n^2t\rfloor), m_1^{(n)}(\lfloor n^2t\rfloor), m_2^{(n)}(\lfloor n^2t\rfloor)\right)_{t\geq 0} \overset{d}{\underset{n\to\infty}{\longrightarrow}} (X_0(t),X_1(t),X_2(t))_{t\geq 0},
\end{equation}
where $(X_0(t),X_1(t),X_2(t))_{t\geq 0}$ is a Wright--Fisher diffusion with mutation rate parameters $(1/2,-1/2,-1/2)$, stopped when one of the latter two coordinates vanishes.  Similar results hold if we initially sample more leaves and look at more branch points, but the limiting description always only holds until the first time a branch point disappears.  As the number of branch points goes to infinity the time for which the limiting description is valid tends to zero.  Moreover, zero is an exit boundary for the coordinates of a Wright--Fisher diffusion that have negative mutation rates, so 
the limiting process does not shed light on how to continue beyond the disappearance of a branch point. 
 In this paper, we give a mechanism through which one can select a new branch point when the original one disappears and, crucially, the process tracking the proportion of leaves in each component remains Markovian.  In fact, we can do this while keeping track of any fixed number of branch points in a way that is projectively consistent over different numbers of branch points being tracked.  We are currently working to understand the consequences of this construction for extending limits as in \eqref{eq:negativewf} beyond the first time a branch point disappears.

\begin{figure}[t]\centering
 \scalebox{.8}{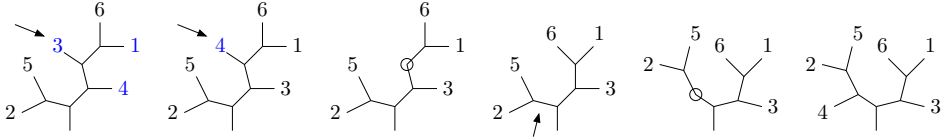}
 \caption{The same down-up move as in Figure \ref{fig:AC_move}, but with label swapping as in Definition \ref{defn:simptreeupdown}. Here, $i=3, a=1, b=4$, so $\wi=4$. We swap labels $i=3$ and $\wi=4$ before deleting the chosen leaf. Label $4$ then regrows in the up-move.\label{fig:downmove}}
\end{figure}

Our strategy is to modify the dynamics of leaf labels in the Aldous chain.  Importantly, this does not affect the evolution of the unlabeled tree; it only affects the labels.  These new label dynamics give rise to a natural mechanism for selecting new branch points about which to partition the tree when old branch points die.

Given a finite set $A$, let $\bT_{A}$ be the space of rooted binary trees with $\#A$ leaves, with the leaves labeled by the set $A$. See Definition \ref{defn:bintree} below.  Similar to the Aldous chain, we introduce a new down-up chain on $\bT_{[n]}$, where $[n]:=\{1,2,3,\dots, n\}$.  

\begin{defn}[Uniform chain]\label{defn:simptreeupdown}
 Fix $n\ge 3$. We define a Markov chain $(T(j))_{j\geq 0}$ on $\bT_{[n]}$ in which each transition comprises a down-move followed by an up-move. Given $T(j) = \ft_n\in\bT_{[n]}$, $j\ge 0$, we randomly construct $T(j+1)$ as follows.
  \begin{enumerate}[(i)]    
  \item Down-move: select the leaf labeled $i\in [n]$ for removal with probability $1/n$. Then
 compare $i$ with the smallest leaf labels $a$ and $b$ in the first two subtrees on the ancestral path from leaf 
        $i$ to the root (with the convention $b=0$ if the path has only one subtree). Let $\wi=\max\{i,a,b\}$. If $\wi\neq i$, swap labels $i$ and $\wi$. Remove the leaf now labeled $\wi$ (which had been labeled $i$) and contract away its parent branch point. Call the resulting tree $\ft_{n-1}\in \bT_{[n]\setminus \{\wi\}}$.
        
      \item Up-move: insert a new leaf $\wi$ at an edge in $\ft_{n-1}$ chosen uniformly at random.
  \end{enumerate}
\end{defn}

Figure \ref{fig:downmove} shows one down-up move of this chain. Note that, if we erase the leaf labels, the above chain is identical to the Aldous chain.  Somewhat surprisingly, although the dynamics of labels are not exchangeable in this down-up chain as they are in the Aldous chain, the two chains have the same stationary distribution.

\begin{thm}\label{prop1}
 The unique invariant distribution for the uniform chain is the uniform distribution on $\tshape{n}$. 
\end{thm}

To state our main result, we require an appropriate notion of projecting a tree $\ft \in \bT_{[n]}$ onto a subtree spanned by $k$ leaves.  We give an informal definition here; for a formal construction see Section \ref{sec:decoration}.  For $\ft \in \bT_{[n]}$ and $k\leq n$, we define $\rho^{\bullet(n)}_k(\ft) := (\fs,f)$, where $\fs$ is the subtree of $\ft$ spanned by $[k]$, reduced by contracting vertices of degree $2$, and $f$ is a weight function on the edge set of $\fs$ that assigns weight to each edge equal to the number of leaves in the corresponding subtree of $\ft$. See Figure \ref{fig:decorated_collapsed}. 
The following is our main result.

\begin{figure}[t]\centering
 \includegraphics[scale=0.68]{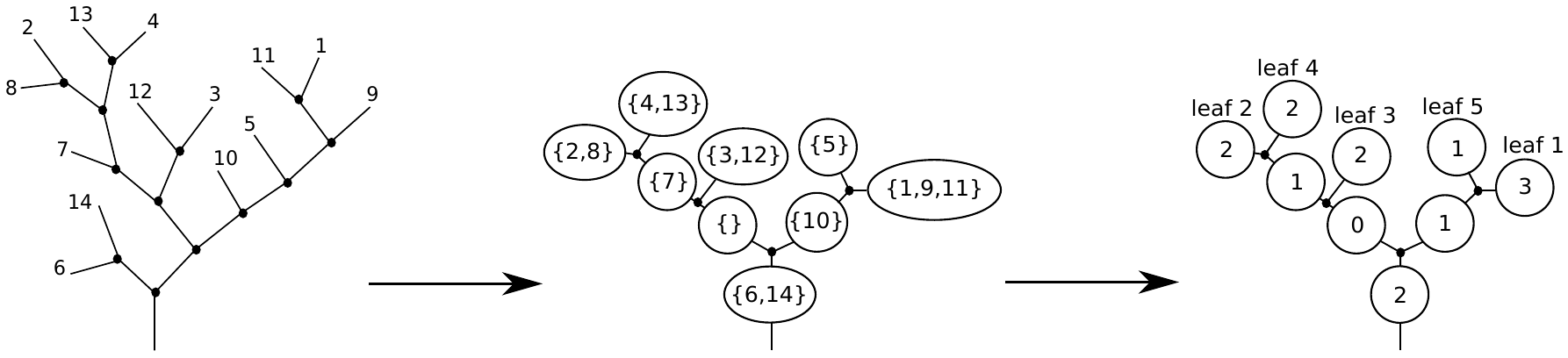}
 \caption{Left: $\ft\in\bT_{[14]}$. Middle: $\rho^{\star(14)}_{5}(\ft)$ -- $k$-tree ($k=5$) with edges marked by labels of collapsed leaves. Right: $\rho^{\bullet(14)}_{5}(\ft)$ -- $k$-tree with edges decorated by leaf mass. See Section \ref{sec:decoration}.\label{fig:decorated_collapsed}}
\end{figure}

\begin{thm}\label{thm projectedchain}
Let $(T(j))_{j\geq 0}$ be a uniform chain running in stationarity.  Then, for $k\leq n$, $(\rho^{\bullet(n)}_k(T(j)))_{j\geq 0}$ is a Markov chain running in stationarity. 
\end{thm}

In Definition \ref{defn:simpdownup} we explicitly describe the transition probabilities of the chains $(\rho^{\bullet(n)}_k(T(j)))_{j\geq 0}$.  We remark that this theorem is not true for $\rho^{\bullet(n)}_k$-projections of the Aldous chain: relatively small examples show that those projected processes are not Markovian.  It is also important that the uniform chain starts in stationarity. 

There is a broad literature on when a function of a Markov chain is again Markov \cite{DiacFill90,Fill92,RogersPitman}, but to our knowledge, Theorem \ref{thm projectedchain} does not directly fit any known framework. In particular, the chains studied here resemble those in the literature on down-up chains on branching graphs \cite{BoroOlsh09,Fulman05,Fulman09b,Kerov96,Pet2009}: (i) each transition comprises a down-move and an up-move, 
and (ii) the set of rooted, leaf-labeled binary trees forms a branching graph.

%


However, as Fulman notes \cite{Fulman09}, despite these similarities, the Aldous chain is not naturally amenable to these down-up methods. 
One of the impediments when working with either the Aldous chain or the uniform chain of Definition \ref{defn:simptreeupdown} is that the same label that is removed in the down-move is reinserted in the up-move.  This dependence between down-move and up-move differs from the standard down-up framework. Our approach combines two well-known criteria for functions of Markov chains to be Markov: intertwining, which arises elsewhere in the down-up literature \cite{BoroFerr14}, and the Kemeny--Snell criterion \cite{KS60}, also known as Dynkin's criterion. See also \cite{Pang17}. The uniform $n$-tree down-up chain $(T(j))_{j\geq 0}$ is intertwined on top of an intermediate Markov chain $(\rho^{\star (n)}_k(T(j)))_{j\ge 0}$ (Figure \ref{fig:decorated_collapsed}, middle panel), which then relates to $(\rho^{\bullet(n)}_k(T(j)))_{j\geq 0}$ via the Kemeny--Snell criterion. 

\subsection{Connection to a conjecture of Aldous}\label{sec:intro:AD}

This paper is a part of a project by the current authors to resolve a long-standing conjecture by Aldous \cite{AldousDiffusionProblem} that the Wright-Fisher limits of \eqref{eq:negativewf}, taken consistently for any number of branch points, record aspects of an underlying ``diffusion on continuum trees.'' This would be a continuum analogue of the Aldous Markov chain on finite trees. Since the stationary law of the Aldous chain, noted in Theorem \ref{prop1}, converges in a scaling limit to the Brownian continuum random tree (BCRT) \cite{Ald-91}, the latter should be the stationary distribution of the conjectured continuum process.

One way to study continuum tree structures is to randomly sample $k$ leaves and consider the subtree generated by these $k$ leaves and the root. 
In \cite{Ald-91}, Aldous constructs the BCRT by taking a limit of such projections that are consistent over $k$. One naturally wonders whether projections of the Aldous chain onto the subtree generated by the first $k$ leaves of the $n$-tree give rise to projectively consistent Markov processes, thus enabling us to take a projective limit. However, this approach quickly runs into non-trivial road blocks. It has been argued elsewhere \cite{Pal13} that, in order to take a continuum limit, one must speed up time by a factor of $n^2$ for the Aldous chain on $n$-leaf binary trees. Consequently, in the limit, leaves become perfectly ephemeral: the entire leaf set dies at every instant. So how can we project onto the ``same'' leaves over time?

To get around this problem, Aldous proposed that one consider branch points rather than leaves. This leads to the challenges raised after \eqref{eq:negativewf} above, which are resolved by Theorem \ref{thm projectedchain}. In light of this theorem, our program for constructing the conjectured Aldous diffusion in future work involves constructing a continuum analogue to (an extension, discussed in the Appendix, of) the chains described in Theorem \ref{thm projectedchain}, for each $k$, and then using projective consistency in $k$ to construct the full continuum-tree-valued process.



\subsection{Applications to statistics of partially classified hierarchical data}

Hierarchically arranged data sets are a common feature in areas such as topic modeling \cite{nCRP,nHDP} used in machine learning and natural language processing, as well as in phylogenetics \cite{felsenstein}. Both types of data can be displayed as a leaf-labeled tree. The leaves of a phylogenetic tree are labeled by a set of taxon (i.e.\ organism) names. The edges have lengths, corresponding to ``time'' between ``mutations'' in the tree. The goal of topic modeling is to classify a given set of documents according to a set of topics and to organize the topics in a hierarchy with more abstract topics near the root. This is the standard way to arrange books in libraries or assign AMS subject classifications (MSC2010) to mathematical research papers. The goal in natural language processing is to \textit{discover} these topics from the words in the given documents. The leaves of the tree may be labeled by the documents themselves. 

The decorated $k$-trees that we study in this article can be used as a model of hierarchically arranged clusters of data or as partially classified hierarchical data. As an example, consider again the AMS subject classifications. One such classification is \textit{Probability theory and stochastic processes} (60XX). This classification has seventeen sub-classifications, including \textit{Limit theorems} (60FXX) which has seven further sub-classifications, including \textit{Large deviations} (60F10) and \textit{None of the above but in this section} (60F99). Multiple papers are grouped into each classification. Each classification corresponds to a vertex of a tree. The leaves (say 60F10) can be thought of as having an associated mass given by the number (or proportion) of papers that are classified by that edge (i.e.\ papers on large deviation). On the other hand, papers tagged 60F99 contain possible further classifications which will all share the same three digits `60F' but are not fully known yet. The number of papers in this last category appear as a mass on the internal edge 60F, representing the number of leaves that would grow out of this edge if each paper received its own, individual classification. A similar data structure emerges when we consider a complete phylogenetic tree of $k$ organisms and classify an additional $m$ organisms according to the branch of the tree from which each has diverged. 

To the best of our knowledge, no thorough statistical analyses have been done regarding these decorated trees, although models in non-parametric Bayesian analysis on nested Chinese restaurant processes \cite{nCRP} and nested hierarchical Dirichlet processes \cite{nHDP} are somewhat related. Our contribution to this research is the introduction of a natural model for decorated binary trees as a projection of large random binary trees (Section \ref{sec:decoration}) and analyzing a sequence of natural Markov chains on the space of such trees that have explicit stationary distributions and are consistent over finer and coarser partial classifications (Section \ref{Sect23}).

\section{Down-up chains on binary trees}

\subsection{Preliminaries}\label{sec:prelim} For $k\in\bN$, let $[k]:=\{1,2,\ldots,k\}$. We denote by $\bT_{[k]}$ the set of rooted binary trees with $k$ leaves labeled $1,\ldots,k$, and more generally for finite $A\subset\bN$ by $\bT_A$ the set of rooted binary trees with $\#A$ leaves labeled by the elements of the label set $A$. Specifically, we have in mind
\[
\begin{split}
\bT_{[1]} &=\left\{\parbox{0.32cm}
{\includegraphics[scale=0.24]{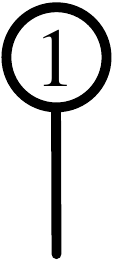}}\right\},\qquad 
\bT_{[2]}=\left\{\parbox{0.6cm}
{\includegraphics[scale=0.24]{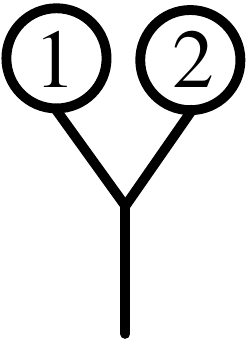}}\right\},\\
\bT_{[3]}&=\left\{\parbox{0.9cm}{\includegraphics[height=1cm]{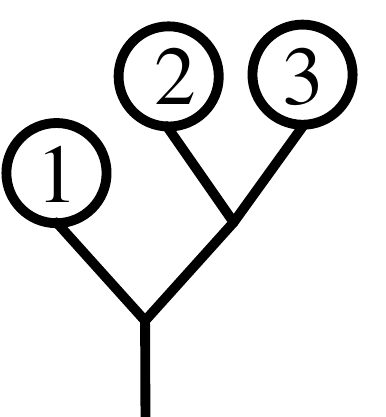}},\quad \parbox{0.9cm}{\includegraphics[height=1cm]{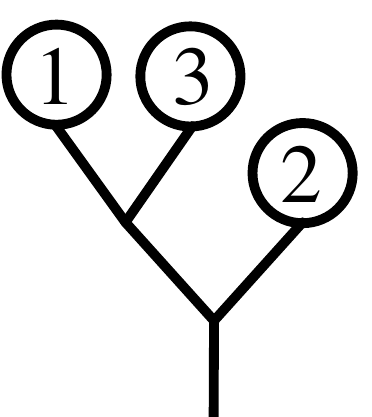}},\quad \parbox{0.9cm}{\includegraphics[height=1cm]{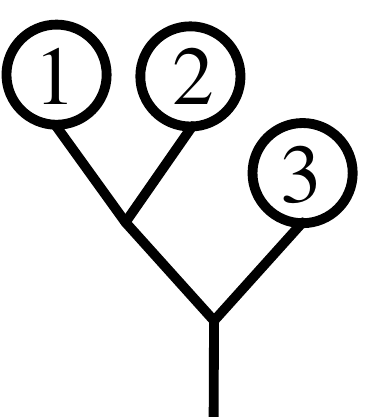}}\right\},\quad \text{etc.,}
\end{split}
\]
where we have ordered subtrees by their least labels to uniquely choose the plane tree representatives. It will be convenient to label each \textit{edge} by the set of leaf labels in the subtree above the edge. This inspires the following definition.

\begin{defn}\label{defn:bintree}
Given a finite subset $A$ of $\bN$, a collection $\ft$ of non-empty subsets of $A$ is called a \em rooted binary tree with leaves labeled by $A$ \em if the following hold. 
\begin{enumerate}[(i)]
\item $\{j\}\in\ft$ for all $j\in A$. These singleton sets are called the \textit{external edges} of $\ft$. Any other set $B$ in $\ft$ with $\# B\ge 2$ will be called an \textit{internal edge}.
\item For all $B_1,B_2\in\ft$, we have $B_1\cap B_2=\varnothing$, or $B_1\subseteq B_2$, or $B_2\subseteq B_1$.
\item For all $B\in\ft\setminus\{A\}$, there is a unique 
    $B^\prime\in\ft$ such that $B^\prime \cap B=\varnothing$ with $B\cup B^\prime \in \ft$; then $B^\prime$ is called the \textit{sibling} edge of $B$, while  $\parent{B} := B \cup B^\prime$ is called the \textit{parent} edge of $B$ and $B^\prime$, and $B$ and $B^\prime$ the \textit{children} of $\parent{B}$. The sibling of the parent of $B$ will be called the uncle of $B$. 
\end{enumerate}
For short, elements of $\bT_{[n]}$ will be called \em $n$-trees \em (and elements of $\bT_{[k]}$ $k$-trees). 
\end{defn}
Conditions (i) and (iii) imply that $A \in \ft$ for all $\ft\in\bT_A$. This notion of a tree $\ft$ is consistent with the graph-theoretic notion of a rooted tree with $\ft$ as the set of edges 
between vertices. In that setting, there is a one-to-one correspondence between non-root vertices and their parental edges, so we can use edge labels to denote non-root 
vertices. For example, the edge labeled $\{j\}$ connects leaf $j$ to its parent. Moreover, we denote the \em root vertex \em by the empty set $\varnothing$ and note that the edge
$A$ is the unique edge incident on it.

With this definition of a tree, we have  
\[
\bT_{[1]}=\Big\{\big\{\{1\}\big\}\Big\},\quad\bT_{[2]}=\Big\{\big\{\{1\},\{2\},[2]\big\}\Big\},
\]
and $\bT_{[3]}$ consists of the three trees
\[
  \big\{\!\{1\},\{2\},\{3\},\{2, 3\},[3]\big\},\big\{\!\{1\},\{2\},\{3\},\{1, 3\},[3]\big\},
  \mbox{ and }\big\{\!\{1\},\{2\},\{3\},\{1, 2\},[3]\big\}.
\]

Given a tree $\ft$ labeled by $A$, the \textit{ancestral line} or \textit{ancestral path} of an edge $B_0$ is the sequence of edges, starting with $B_1=\parent{B_0}$, and, 
inductively, $B_{k+1}=\parent{B_k}$, until we reach $B_m=A$, for some $m$. This will be referred to as the \textit{path from $B_0$ to the root}. Terms such as the 
\textit{most recent ancestor} of $B_0$ satisfying some property will be naturally defined by this notion of ancestry. If $B_0=\{j\}$ for some leaf labeled $j\in A$, there is a 
natural \textit{spinal decomposition} of the tree $\ft$ into the sequence of subtrees rooted along this ancestral line. That is, consider the vertices labeled by 
$B_1, B_2, \ldots, B_m$. At each $B_k$, there is a binary tree, rooted at $B_k$, whose leaves are labeled by $B_{k-1}'$. Since $\{j\} ,  B_{k-1}',\; k=1,2,\ldots, m$, are 
disjoint and exhaust $A$, this gives a decomposition of $\ft$. We will refer to the \textit{$k$th subtree on this path} to refer to the subtree rooted at $B_k$. 

If $C\subset A$ and $\ft$ is a tree labeled by $A$, we define $\ft\cap C$ to be the collection of non-empty elements of the set $\left\{ B\cap C, B \in \ft  \right\}$. For $\ft\in\bT_A$ and $j\in A$, define 
\[
\fs:=\ft- j:=\ft\cap(A\setminus\{j\})=\{B\setminus\{j\},B\in\ft\}\setminus\{\varnothing\}
\]
to be the tree $\ft$ with leaf $j$ removed. Note that the parent $\sparent{j}$ and the sibling $\{j\}'$ in $\ft$ have been identified as one edge in $\fs$, and that all edges in $\ft$ on the ancestral line of leaf $j$ have lost $j$ from their label sets in $\fs$. 

If we are given an edge $S\in\fs$, we can add the leaf $j$ as a sibling of $S$ by splitting $S$ into two edges $S$ and $S\cup\{j\}$, and adding $j$ to the label set of all edges on the ancestral line of $S$ in $\fs$. The resulting tree, denoted by  $\fs+(S,j)$, is therefore 
\[
\{B\in\fs\colon B\cap S=\varnothing\}\cup\{B\in\fs\colon B\subseteq S\}\cup\{B\cup\{j\}\colon B\in\fs\mbox{ and }S\subseteq B\}\cup\{\{j\}\}.
\]

We refer to the operation $(\ft,j)\mapsto\ft- j$ as the \em deletion of leaf $j$ from $\ft$\em, and to the operation $(\fs,S,j)\mapsto\fs+(S,j)$ as the \em insertion of leaf $j$ into $\fs$ on edge $S$.\em  

It is now easy to see that for each fixed $j\in A$, we have 
$$\bT_A=\{\fs+(S,j)\colon \fs\in\bT_{A\setminus\{j\}},S\in\fs\},$$
and that the representation of an element of $\bT_A$ as $\fs+(S,j)$ is unique so that, inductively, we obtain the well-known enumeration formula for numbers  
\eq\label{eq:enumtree}
\#\bT_A=\prod_{1\le i\le \#A-1}(2i-1)
\en
of rooted binary trees with $\#A$ labeled leaves.

\subsection{Proof of Theorem \ref{prop1} and generalization to Ford's alpha model trees} 

The proof of Theorem \ref{prop1} requires the concept of a \em binary tree growth process \em (see \cite{PW09}). This is a sequence  
$(T_n,n\ge 1)$ of trees $T_n\in\bT_{[n]}$ with the consistency property that $T_n- n=T_{n-1}$ for all $n\ge 2$. As this entails that $T_n=T_{n-1}+(S_{n-1},n)$ for some $S_{n-1}\in T_{n-1}$, the conditional distribution of $T_n$ given $T_{n-1}$ is specified by the conditional distribution of $S_{n-1}$ given $T_{n-1}$. 


\begin{defn}[R\'{e}my's uniform growth process]\label{defn:remy} Let $T_1$ be the unique element in $\bT_{[1]}$. For every $n\ge 2$,
  let $T_n:=T_{n-1}+(S_{n-1},n)$, where conditionally given $T_{n-1}$, the random edge $S_{n-1}$ is uniformly distributed on 
  the set of edges of $T_{n-1}$. Then $(T_n,n\ge 1)$ is called R\'emy's uniform binary tree growth process or just \em uniform growth 
  process \em \cite{remy}. 
\end{defn}

In particular, $T_n$ is uniformly distributed on $\bT_{[n]}$. We will denote the uniform distribution on $\tshape{n}$ by $q_{n,1/2}$. The parameter $1/2$ alludes to a one-parameter family of tree growth process, due to Ford \cite{For-05}. See also Pitman and Winkel \cite{PW09}. 

\begin{defn}[Ford's alpha growth process]\label{defn:ford}
  Fix $\alpha\in [0,1]$. Let $T_1$ and $T_2$ be the unique elements in $\bT_{[1]}$ and $\bT_{[2]}$, respectively. Now, 
  inductively, given $T_{n-1}$ for some $n\ge 3$, assign a weight $1-\alpha$ to each external edge and a weight $\alpha$ to 
  each internal edge of $T_{n-1}$, let $S_{n-1}$ be a random edge of $T_{n-1}$ whose distribution is proportional to these weights, and 
  define $T_n= T_{n-1} + \left( S_{n-1}, n \right)$. Then $\left( T_n,\, n\ge 1  \right)$ is called Ford's alpha model 
  growth process or just \em alpha growth process\em. 
\end{defn}

For $\alpha=1/2$, this is the uniform growth process. We denote the distribution of $T_n$ under the alpha growth process by
$q_{n,\alpha}$. The case for $\alpha=0$ is also known as the Yule model while $\alpha=1$ is called the Comb. 
Crucially, the leaf labels of $T_n$ are only exchangeable in the uniform $\alpha=1/2$ case.
Naturally, one can ask if Theorem \ref{prop1} generalizes to other values $\alpha\in[0,1]$. The answer is yes and no. A slightly different Markov chain given below generalizes to all $\alpha \in [0,1]$. However, for $\alpha=1/2$, this does not give us the uniform chain of Definition \ref{defn:simptreeupdown}, which is our main interest. 

\begin{defn}[Alpha chain]\label{defn:ntreeupdown} Fix $n\ge 3$. We define a Markov chain $(T(j))_{j\ge 0}$ on $\bT_{[n]}$ in which each 
transition comprises a down-move followed by an up-move. Given $T(j)=\ft_n \in \bT_{[n]}$, $j\ge 0$, we randomly construct $T(j+1)$,
as follows. 
  \begin{enumerate}[(i)]    
    \item Down-move: select the leaf labeled $i\in\{1,\ldots,n\}$ for removal with probability $1/n$. Compare $i$ with the smallest 
      leaf labels $a$ and $b$ in the first two subtrees on the ancestral path from leaf $i$ to the root ($b=0$ if 
      the path has only one
      subtree). Let $\wi=\max\{i,a,b\}$. Swap labels $i$ and $\wi$ and remove the leaf now labeled $\wi$ (which had been labeled $i$).
      Relabel leaves $\wi+1, \ldots, n$ by $\wi, \ldots, n-1$, respectively. This produces an element $\ft_{n-1}$ in 
      $\tshape{n-1}$.    
    \item Up-move: insert a new leaf labeled $n$ according to the alpha growth process.
  \end{enumerate}
\end{defn}

\begin{thm}\label{prop4} For each $n\ge 3$ and each $\alpha\in [0,1]$, the unique invariant distribution for the alpha chain on $n$-trees is $q_{n,\alpha}$. 
\end{thm}

\begin{figure}[t]
\begin{tabular}{c@ {\hskip 0.5 in} c@ {\hskip 0.5 in} c@ {\hskip 0.5 in} c}

\begin{tikzpicture}[scale=0.4]
 \node {(1)} [grow'=up]
 child {[fill] circle (2pt)
 	child{[fill] circle (6 pt)} 
	child [black] {[fill] circle (2pt) edge from parent[red, thick]
		child [black] {node {$\wi$} edge from parent[black, thin]}
		child [black] {node {$i$} edge from parent[red, thick]}
	}
	};
\end{tikzpicture} 

	& 

\begin{tikzpicture}[scale=0.4]
 \node {(2)} [grow'=up]
 child {[fill] circle (2pt)
 	child{[fill] circle (6 pt)} 
	child [black] {[fill] circle (2pt) edge from parent[red, thick]
		child [black] {node {$i$} edge from parent[black, thin]}
		child [black] {node {$\wi$} edge from parent[red, thick]}
	}
	};
\end{tikzpicture} 

& 
	
\begin{tikzpicture}[scale=0.4]
 \tikzstyle{level 2}=[sibling distance =20 mm]
 \tikzstyle{level 3}=[sibling distance=10 mm]
 \node {(3)} [grow'=up]
 child {[fill] circle (2pt)
  	child{[fill] circle (2 pt)
		child{[fill] circle (2 pt)}
		child{[fill] circle (6 pt)}
		} 
	child [black] {[fill] circle (2pt) edge from parent[red, thick]
		child [black] {[fill] circle (2 pt) edge from parent[thin]
			child [black] {[fill] circle (2 pt)}
			child [black] {node {$i$} edge from parent[black]}
			}
		child [black] {node {$\wi$} edge from parent[red, thick]}
	}
	};
\end{tikzpicture}	

	&

\begin{tikzpicture}[scale=0.4]
 \tikzstyle{level 2}=[sibling distance =20 mm]
 \tikzstyle{level 3}=[sibling distance=10 mm]
 \node {(4)} [grow'=up]
 child {[fill] circle (2pt)
 	child{[fill] circle (2 pt)
		child{[fill] circle (2 pt)}
		child{[fill] circle (6 pt)}
		} 
	child{[fill] circle (2pt)
		child {[fill] circle (2 pt)}
		child [black] {node {$i$}}
	}
	};
\end{tikzpicture}

\end{tabular}	
\caption{From left to right, we swap $i$ and $\wi$ in (2) and insert two new leaves in (3). No leaves attach to the red edges. The final tree in (4) is obtained by dropping $\wi$ and the red edges.}
\label{fig:swapgrowdrop}
\end{figure}
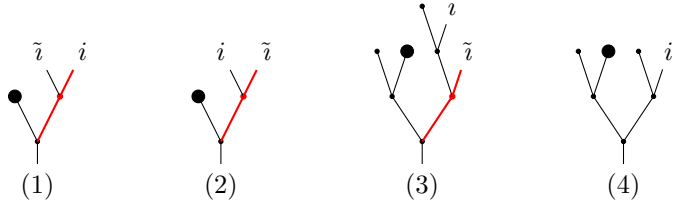

Key to the proof of Theorems \ref{prop1} and \ref{prop4} is a consistency property of the trees $T_n$, $n\ge 1$, in the alpha growth 
process under the down-move part of the alpha chain: 

\begin{lemma}\label{lem:downinv} Consider the tree $T_n$ in the alpha growth process. Fix $1\le i \le \wi \le n$, with $\wi \ge 2$. 
  Let $E_{i,\wi}$ be the event that $\wi=\max\left( i,a,b  \right)$ where $a$ is the smallest label in the first subtree on the 
  ancestral path from leaf $i$ to the root in $T_{n}$, and $b$ is the smallest label in the second subtree on that same path. Consider 
  the conditional probability distribution $\pP\left(\cdot \mid E_{i,\wi}\right)$. Swap leaves $i$ and $\wi$, drop leaf $\wi$, and relabel 
  leaves $\wi+1, \ldots, n$ by $\wi,\ldots, n-1$, respectively. Then the resulting tree, $\widetilde{T}_{n-1}$, under 
  $\pP\left( \cdot \mid E_{i,\wi}  \right)$, is distributed as $q_{n-1,\alpha}$.
\end{lemma}

\begin{proof}
  Consider an alpha growth process $(T_k,k\ge 1)$. First we claim that the event $E_{i,\wi}$ is independent of $T_{\wi-1}$. This 
  statement is trivial for $\wi=2$, when $T_{\wi-1}$ is deterministic. Henceforth, assume $\wi\ge 3$. Now $E_{i,\wi}$ can be expressed as follows. 
  \begin{enumerate}[(i)]
    \item If $i < \wi$, consider the ancestral path of leaf $i$ in $T_{\wi-1}$. There are at least two edges on this path (say, $\{i\}$ 
      and $[\wi-1]$), and $E_{i,\wi}$ is equivalent to the 
      event that $\wi$ gets attached to one of the first two edges on the ancestral path from leaf $i$ to the root, and none of the 
      leaves with label in $\{ \wi+1, \ldots, n \}$ gets attached to the first two edges on the ancestral path of leaf $i$ in $T_{\wi}$. 
    \item Otherwise $i=\wi$. Then $E_{\wi,\wi}$ is the event that none of the leaves with label in $\{ \wi+1, \ldots, n \}$ gets 
      attached to the first two edges on the ancestral path from $\wi$ to the root in $T_{\wi}$.
  \end{enumerate}
  In case (i) the leaf $\wi$ has to get attached to a given pair of internal and external edges, and, in both cases (i) and (ii), all   
  higher labeled leaves are conditioned not to get attached to a given pair of internal and external edges. Hence, in both cases, 
  $E_{i,\wi}$ is independent of $T_{\wi-1}$, and  $T_{\wi-1}$ remains distributed as $q_{\wi-1, \alpha}$ under the conditional law 
  $\pP\left( \cdot \mid E_{i, \wi}  \right)$,.

  Now, consider $T_n$ under $\pP\left( \cdot \mid E_{i, \wi}  \right)$. Swapping $i$ and $\wi$, and dropping the new leaf labeled $\wi$  
  and all the higher-labeled ones, we still obtain $T_{\wi-1}$, which is distributed as $q_{\wi-1, \alpha}$. See Figure 
  \ref{fig:swapgrowdrop}. Under $\pP\left( \cdot \mid E_{i, \wi}  \right)$, given $T_{\wi-1}$, consider how new leaves get attached to 
  form $\widetilde{T}_{n-1}$. The old leaf $\wi$ never gets attached, therefore the two red edges in Figure \ref{fig:swapgrowdrop} (2) 
  do not exist anymore. Leaves $(\wi+1, \ldots,  n)$ have been relabeled $(\wi, \ldots, n-1)$. Consider the distribution of the edges 
  $(S_\wi, \ldots, S_n )$ where these successive leaves get attached. Each new leaf still gets attached to every internal edge with 
  probability proportional to $\alpha$ and to an external edge with probability proportional to $(1-\alpha)$. Hence, we still have an 
  alpha growth process starting from $T_{\wi-1}$. Since the law of $T_{\wi-1}$ is $q_{\wi-1,\alpha}$, this completes the proof. 
\end{proof}

\begin{proof}[Proof of Theorems \ref{prop1} and \ref{prop4}] 
  The uniqueness of the invariant distribution in both theorems is immediate since the Markov chain is finite and 
  either is irreducible (for $\alpha\in(0,1)$) or has a unique recurrent communicating class (for $\alpha\in\{0,1\}$). 

  Consider Theorem 
  \ref{prop4} and one step of the Markov chain. Suppose $T_n(0) \in \tshape{n}$ is distributed as $q_{n,\alpha}$ and let $I$ be
  uniformly chosen from $[n]$, independent of $T_n(0)$. For a given pair of $1\le i\le \wi\le n$, $\wi\ge 2$, let $F_{i, \wi}$ be the 
  event that $I=i$ is selected for the down-move and swapped with label $\wi$. 
  Condition on $F_{i, \wi}$. Under this conditional distribution, apply Lemma \ref{lem:downinv} to get a tree $\widetilde{T}_{n-1}(0)$, 
  which is distributed as $q_{n-1, \alpha}$. Add a leaf labeled $n$ according to the alpha growth process to get a tree ${T}_n(1)$ that 
  is distributed according to $q_{n,\alpha}$. Since the conditional distribution of $T_n(1)$ is $q_{n,\alpha}$, given $F_{i,\wi}$, 
  irrespective of the pair $(i, \wi)$, its unconditional distribution is also $q_{n,\alpha}$. This completes the proof of Theorem 
  \ref{prop4}. 

  To prove Theorem \ref{prop1}, given $F_{i, \wi}$, get $T_{n}(1)\sim q_{n,1/2}$ as above. Relabel the new leaf $n$ by $\wi$ and leaves 
  labeled $(\wi, \ldots, n-1 )$ by $( \wi+1,\ldots, n)$. By exchangeability of labels under $q_{n,1/2}$, this new tree is again uniformly 
  distributed. But, of course, this is the same as simply inserting a new leaf labeled $\wi$ while leaving the higher-labeled leaves 
  unchanged. Integrate over $F_{i, \wi}$, as before, to obtain the result.   
\end{proof}

Consider again the uniform chain from Definition \ref{defn:simptreeupdown}. Let $\widetilde{I}$ denote the random label that is inserted
in the first up-move. We say that the leaf $\widetilde{I}$ is resampled. As an application of our proof technique, we show that the 
distribution of the resampled label is not uniform over $[n]$. This is a special case of the alpha chain (Definition 
\ref{defn:ntreeupdown}) where $\widetilde{I}$ now refers to the label removed in the down-move. 

\begin{cor}\label{cor:resampdist}
Suppose $T_n$ is distributed according to $q_{n,\alpha}$ on $\tshape{n}$. Consider one step of the alpha chain starting at $T_n$ and let $\widetilde{I}$ denote the random leaf label removed in the down-move. Then
\eq\label{eq:distwi}
\pP\left(  \widetilde{I}=\wi \right)=\begin{dcases}
\frac{1}{n(n-1-\alpha)}, & \text{if $\wi=2$,}\\
\frac{2\wi -2-\alpha}{n(n-1-\alpha)}, & \text{if $3\le \wi \le n$}.
\end{dcases}
\en
In particular, for the stationary uniform chain, the case $\alpha=1/2$ of the above gives the distribution of the resampled leaf label.
\end{cor}

\begin{proof} Consider the event $E_{i,\wi}$ from Lemma \ref{lem:downinv}. First consider the special case of $\alpha=1/2$ which is simpler. Assume $\wi \ge 3$ and $i < \wi$. Then, it follows from alternative expression (i) of $E_{i, \wi}$ that
\[
 \pP\left(  E_{i,\wi} \right)= \frac{2}{2\wi-3} \prod_{j=\wi}^{n-1} \left(1 - \frac{2}{2j-1} \right)=\frac{2}{2\wi-3} \prod_{j=\wi}^{n-1} \left(\frac{2j-3}{2j-1} \right)=\frac{2}{2n-3}.
\]
Similarly, for $i=\wi$, $\displaystyle
\pP\left(  E_{\wi,\wi} \right)= \prod_{j=\wi}^{n-1} \left(1 - \frac{2}{2j-1} \right)=\prod_{j=\wi}^{n-1} \left(\frac{2j-3}{2j-1} \right)=\frac{2\wi-3}{2n-3}$.

Let $I$ be the leaf chosen uniformly in the down-move, independent of $T_n(0)$. Denote by $\widetilde{I}$ the random leaf label that gets removed in the down-step when $I$ is 
selected. Integrating over the choice of $\{I=i\}$, $i\in [n]$, we get
\[
 \pP\left( \widetilde{I}=\wi  \right)= \frac{1}{n}\sum_{i=1}^{\wi} \pP\left( E_{i,\wi}  \right)=\frac{1}{n}\left[ \frac{2(\wi-1)}{2n-3} + \frac{2\wi-3}{2n-3}  \right]=\frac{4\wi-5}{n(2n-3)}, \quad \wi\ge 3.
\]
Then, when $\wi=2$, $\displaystyle\pP\left( \widetilde{I}=2  \right)= 1-\sum_{j=3}^n\pP\left(\widetilde{I}=j\right)= \frac{2}{n(2n-3)}$.

For any other $\alpha\in(0,1)$, the proof is very similar. On the event $E_{i,\wi}$, new leaves are forbidden to get attached to one internal and one external edge. The probability that leaf $j+1$ gets attached to any given internal edge is $\alpha/(j-\alpha)$ and the probability that it gets attached to a given external edge is $(1-\alpha)/(j-\alpha)$. Thus, for $\wi\ge 3$,
\[
\pP\left(  E_{i,\wi} \right)= \begin{dcases}
\frac{1}{\wi-1-\alpha}\prod_{j=\wi}^{n-1}\left( 1 - \frac{1}{j-\alpha}\right)=\frac{1}{n-1-\alpha}, & \text{for $1\le i < \wi$},\\
\prod_{j=\wi}^{n-1}\left( 1 - \frac{1}{j-\alpha}\right)=\frac{\wi-1-\alpha}{n-1-\alpha}, & \text{for $\wi=i$}. 
\end{dcases}
\]
Summing up, 
\[
\pP(\widetilde{I}=\wi)=\frac{1}{n}\left[ \frac{\wi-1}{n-1-\alpha} + \frac{\wi-1-\alpha}{n-1-\alpha}  \right]=\frac{2\wi-2-\alpha}{n(n-1-\alpha)}, \quad \wi\ge 3,
\]
and, as above, $\pP(\widetilde{I}=2)=1/n(n-1-\alpha)$. This completes the proof.
\end{proof}

\section{Collapsing a tree and decorating edges with masses}\label{sec:decoration} 


\begin{defn}[Decorated $k$-trees of mass $n$]\label{defn:edgemass}
Fix integers $n \ge k\ge 1$. For $\fs\in\bT_{[k]}$, denote by $\cN_n^\fs$ the set of functions $f\colon\fs \rightarrow (\bN\cup\{0\})$ that satisfy two conditions
\begin{enumerate}[(i)]
\item $\sum_{B\in \fs} f(B)=n$ and
\item $f\left( \{ j\} \right) \ge 1$ for each $j \in [k]$.
\end{enumerate}
Then the set of \em $k$-trees decorated with total edge mass $n$, \em or \em decorated $k$-trees \em for short, is given by 
\[
\bT_{[k]}^{\bullet(n)}:=\bigcup_{\fs\in\bT_{[k]}}  \{\fs\} \times \cN_n^\fs.
\]
We refer to an element $\ft^\bullet=\left( \fs, (x_1,\ldots,x_k, (y_B)_{B\in\fs\colon\#B\ge 2})\right)\in \{\fs\} \times \cN_n^\fs$
as a tree $\fs$ with integer masses attached to every edge. The labeled external edges have positive masses $x_1, \ldots, x_k$. Each internal edge $B$ has a nonnegative mass $y_B$. We say that $\ft^\bullet$ has \em tree shape \em $\fs$.
\end{defn}

A natural way to obtain decorated $k$-trees is to project the leaf mass from an $n$-tree to the subtree spanned by the first $k$ leaves. We describe this concept formally. 

Let $\ft\in\bT_{[n]}$ and $k\le n$. We will associate a decorated tree $\ft^\bullet\in\bT_{[k]}^{\bullet(n)}$, as follows. Define $\fs:=\ft\cap[k]$. For every $j \in [n]$, consider the ancestral line of the leaf $j$, $\sparent{j}=B_1^{(j)} \subseteq B_2^{(j)} \subseteq \cdots \subseteq B_{m_j}^{(j)}=[n]$. Additionally, let $B_0^{(j)}=\{j\}$. In this ordering let $B_j$ be the most recent ancestor $B_i^{(j)}$ of $\{j\}$, including $\{j\}$ itself, such that $B_i^{(j)} \cap [k] \neq \varnothing$. Let $\mu$ be the measure on $\ft$, of total mass $n$, that assigns mass 1 to each of the $n$ leaves. Consider the projection maps $\pi\colon[n]\rightarrow\fs$ given by $j \mapsto B_j\cap [k]$. The push-forward of $\mu$ by this projection assigns a nonnegative integer value to every edge of $\fs$ producing a decorated $k$-tree $\ft^\bullet$ as in Definition \ref{defn:edgemass}. Since $B_j=\{j\}$ for each $j\in[k]$, the projected mass on each leaf of $\fs$ is at least one. The total mass is $n$. We denote the map that assigns the decorated $k$-tree $\ft^\bullet\in\bT_{[k]}^{\bullet(n)}$ to the tree $\ft\in\bT_{[n]}$ by $\rho^{\bullet(n)}_k\colon\bT_{[n]}\rightarrow\bT_{[k]}^{\bullet(n)}$,  $\ft\mapsto\rho_k^{\bullet(n)}(\ft)=\ft^\bullet$. 

For later use we also introduce an intermediate stage between the full $n$-tree and the decorated $k$-tree with mass $n$ obtained from it. 

\begin{figure}[t]
\begin{tikzpicture}[scale=0.40, grow=up]
\node at (0,-0.8) [circle] {$\parent{B}$};
\tikzstyle{level 1}=[sibling distance=30 mm]
 \tikzstyle{level 2}=[sibling distance =35 mm]
 \tikzstyle{level 3}=[sibling distance=25 mm]
	\coordinate
		child {node {1}}
		child {[fill] circle (2pt) edge from parent [red]
			child {[fill] circle (2pt) edge from parent [black]
				child{ node {8} }
				child{ node {9} }
				}
			child {[fill] circle (2pt) 
				child{node[black] {5} edge from parent [black]} 
				child{ [fill] circle (2pt) 
					child {node[black] {3} edge from parent [black]} 
					child {node[black] {2} edge from parent [black]}
					node [black, right] {B}
					}
				}
		};	
\end{tikzpicture}\hspace{1cm}
\begin{tikzpicture}[scale=0.40, grow=up]
\tikzstyle{level 1}=[sibling distance=30 mm]
 \tikzstyle{level 2}=[sibling distance =35 mm]
 \tikzstyle{level 3}=[sibling distance=25 mm]
	\node [shape=circle] {$\varnothing$}
		child {[fill] circle (2pt) edge from parent [red]
			child {[fill] circle (2pt) edge from parent [black]
				child{ node {8} }
				child{ node {9} }
				}
			child {[fill] circle (2pt) 
				child{node[black] {5} edge from parent [black]} 
				child{ node[black] {1}	}
				}
		};
\end{tikzpicture}\hspace{1cm}			
\begin{tikzpicture}[scale=0.40, grow=up]
\tikzstyle{level 1}=[sibling distance=30 mm]
 \tikzstyle{level 2}=[sibling distance =35 mm]
 \tikzstyle{level 3}=[sibling distance=25 mm]
	\node [shape=circle] {$\varnothing$}
		child {[fill] circle (2pt) edge from parent [red]
			child {[fill] circle (2pt) edge from parent [black]
				child{ node {3} }
				child{ node {4} }
				}
			child {[fill] circle (2pt) 
				child{node[black] {2} edge from parent [black]} 
				child{ node[black] {1}	}
				}
		};			
\end{tikzpicture}
\caption{
On the left, consider the subtree generated by leaves $[3]$. The three red edges contract to an internal edge $B$ in the $3$-tree. The unranked internal structure $\fv_B$ is shown in the middle. The vertex $\stackrel{\leftarrow}{B}$ is replaced by the root, while $B$ itself is replaced by a leaf labeled $1$. The internal structure $\int(\fv_B)$ is on the right.}
\label{fig:internal_structure}
\end{figure}
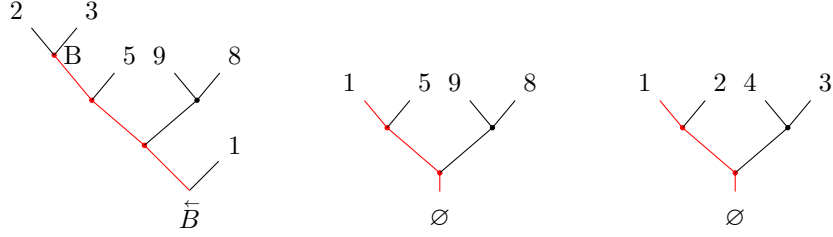

\begin{defn}[Collapsed $n$-trees with $k$ leaves]\label{defn:collapse} 
Consider an $n$-tree $\ft$ and a decorated $k$-tree $\ft^\bullet$ obtained by the projection map $\pi:[n] \rightarrow \fs=\ft \cap [k]$. Let $\pi^{-1}$ be the pre-image function of $\pi$. That is, for every edge $B$ of $\fs$, we associate a subset of $[n]$ given by $\pi^{-1}(B)$. Thus $\mu(B)= \# \pi^{-1}(B)$ is the number of elements in $\pi^{-1}(B)$. The resulting object 
\[
\ft^\star:=\left( \fs, \left(  \pi^{-1}(B)   \right)_{B\in \fs}  \right)
\]
will be called a \emph{collapsed $n$-tree with $k$ leaves} or just a \emph{collapsed $n$-tree}. Denote the set of all collapsed $n$-trees with $k$ leaves (as $\ft$ varies over $\tshape{n}$) by $\bT_{[k]}^{\star(n)}$. The map from $\tshape{n}$ to $\bT_{[k]}^{\star(n)}$ that takes $\ft$ to $\ft^\star$ will be denoted by $\rho_k^{\star(n)}$. That is, $\ft^\star= \rho_k^{\star(n)}\left( \ft \right)$. 
\end{defn}

See Figure \ref{fig:decorated_collapsed} for examples of decorated and collapsed trees obtained by projecting a tree with $n=14$ leaves on the subtree spanned by the first $k=5$ leaves. 




Fix $1\le k < n$ and let $\ft^\star=\left( \fs, \left( \pi^{-1}(B)  \right)_{B\in \fs}  \right) \in \bT_{[k]}^{\star(n)}$ be obtained by collapsing $\ft\in \tshape{n}$. The collection of subsets $\left( \pi^{-1}(B)  \right)_{B\in \fs}$ is an ordered partition of $[n]$, ordered by the edges of $\fs$. Fix an edge $B$ of $\fs$. Let $\fr$ be the subtree of $\ft$ spanned by the leaves $1,2,\ldots, k$, including any degree-two vertices. Label the edges of $\fr$ by their corresponding labels in $\ft$. Consider the set of edges $S$ of $\fr$ such that $S\cap [k]=B$ and consider the vertices corresponding to those edges. Then, there is a sequence of subtrees of $\ft$ with leaf labels in $\pi^{-1}(B)$ that are rooted at those vertices. Along with the edges $S$ of $\fr$ such that $S\cap[k]=B$, this gives us a binary tree $\fv_B$ rooted at the vertex corresponding to $\parent{B}$, the parent of $B$ in $\fs$. Note that, if $B$ is an internal edge of $\fs$, the vertex corresponding to $B$ becomes a new leaf of $\fv_B$. We will always label this leaf by $1$ in $\fv_B$. Other leaves of $\fv_B$ retain their labels from $\ft$, including the case when $B$ is itself an external edge in $\fs$. In particular, the leaf set of $\fv_B$ is $\pi^{-1}(B)$ if $B$ is external in $\fs$ and $\pi^{-1}(B)\cup \{1\}$, if $B$ is internal. These labelings are local in the sense that leaves labeled $1$ appear in $\fv_B$ for all internal edges $B$ in $\fs$. See Figure \ref{fig:internal_structure} for an example where $B$ is internal.

\begin{defn}[Internal structure]\label{defn:internal_structure}
Let $A=\{a_1,\ldots,a_k\}\subseteq [n]$, $k\le n$, where the elements are arranged in increasing order $a_1 < a_2 < \cdots < a_k$. Then we associate with any tree in $\bT_A$ the tree in $\bT_{[k]}$ obtained by relabeling each leaf $a_i$ by its rank $i$. In the case where we apply this relabeling to $\fv_B$ for $\ft\in\bT_n$, $B\in\ft\cap[k]$, we use the function notation $\int(\fv_B)\in\bT_{[\#\pi^{-1}(B)]}\cup\bT_{[\#\pi^{-1}(B)+1]}$. We will refer to this tree as the \textit{internal structure} of $\rho_k^{\bullet(n)}(\ft)$ supported on $B$, while $\fv_B$ itself will be called the \emph{unranked internal structure} of $\rho_k^{\bullet(n)}(\ft)$ supported on $B$.
\end{defn}


We write $\widetilde{q}_{n,\alpha}$ for the law of the $n$-leaf tree formed by Ford growth, as in Definition \ref{defn:ford}, 
with the exception that new leaves select to grow on edge $\{1\}$, specifically, with probability proportional to $\alpha$, rather than $1-\alpha$. 
Note that $\widetilde{q}_{n,\alpha} = q_{n,\alpha}$ for $\alpha=\frac12$.


\begin{lemma}[Spatial Markov property of $q_{n,\alpha}$]\label{lem:spatialMarkov}
 Fix $\alpha\in [0,1]$. If $T_n \sim q_{n,\alpha}$ conditioned on $\rho_k^{\bullet(n)}\left( T_n \right) = \left( \fs, \left( \mu(B)  \right)_{B\in \fs}  \right) $, then the random internal structures $\left( \int\left(V_B\right),\; B\in \fs   \right)$ 
 of $\rho_k^{\bullet(n)}(T_n)$ are mutually independent, jointly independent of $\left( \pi^{-1}(B)  \right)_{B\in \fs}$, and each $\int\left(V_B\right)$ is distributed according to $q_{\mu(B), \alpha}$ if $B$ is an external edge, or $\widetilde{q}_{\mu(B)+1, \alpha}$ otherwise. In particular, these structures remain independent when conditioned on $\rho^{\star(n)}_k(T_n) = \ft^\star = \big( \fs, \left( \pi^{-1}(B) \right)_{B\in \fs}  \big)$. 
\end{lemma}

\begin{proof}
Consider Ford's alpha growth process starting from $\fs$. It is clear from the definition that, given that a label $i$ is in $\pi^{-1}(B)$ for some $B \in \fs$, the location where it grows in $V_B \cap [i-1]$ is independent of the internal structures of all other subtrees growing on other edges. Moreover, the probability that $i$ will grow out of a given edge in $V_B \cap [i-1]$ is as in Definition \ref{defn:ford} of Ford growth, with the exception that if $B$ is internal, then the external edge $\{1\}$ in $V_B$, which corresponds to an internal edge of $T_n\cap [i-1]$, is selected with probability proportional to $\alpha$, instead of $1-\alpha$. Therefore, given $\ft^\star$, each $\int\left(V_B\right)$ is distributed as $q_{\mu(\{i\}),\alpha}$ if $B=\{i\}$ is external or as $\widetilde{q}_{\mu(B)+1, \alpha}$ if $B$ is internal, and these internal structures are jointly independent of one another. Since, conditionally given $\left(\mu(B),B\in\fs\right)$, the joint distribution of the internal structures is independent of the leaf labels $\left( \pi^{-1}(B),\; B\in \fs \right)$, the statement of the lemma follows.
\end{proof}

\begin{lemma}[Regenerative spinal compositions]\label{lem:topbeadkernel}
 Fix $\alpha\in (0,1)$ and consider $T_n\sim\widetilde{q}_{n,\alpha}$. Let $U$ denote the first subtree along the spine from $\{1\}$ to the root and $M_1$ the number of leaves in $U$. Then
 \begin{equation}\label{eq:returnalpha}
  \pP(M_1 = m) = \delta_{\alpha}(n:m) := \alpha \combi{n}{m} \frac{\Gamma(m-\alpha)\Gamma(n-m+\alpha)}{\Gamma(1-\alpha)\Gamma(n+\alpha)},\quad m\in [n].
 \end{equation}
 Moreover, if $V$ is $T_n$ with $U$ removed but leaf $1$ retained, then $\int(U)$ is conditionally independent of $\int(V)$ given $M_1$, with $\int(U)\sim q_{M_1,\alpha}$ and $\int(V)\sim \widetilde{q}_{n-M_1,\alpha}$.
\end{lemma}

We remark that, when $\alpha=\frac12$, \eqref{eq:returnalpha} is also the probability that the first return to the origin of a simple random walk bridge of length $2n$ occurs at time $2m$:
 \begin{equation}\label{eq:return}
  \delta(n:m) := \delta_{1/2}(n:m) = \frac{1}{2m-1} \combi{2m}{m} \combi{2n-2m}{n-m}\Big/\combi{2n}{n}, \quad m\in [n].
 \end{equation}

\begin{proof}
 Think of $T_n$ as growing according to Ford's alpha growth, with the modification noted above Lemma \ref{lem:spatialMarkov}, and consider how subtrees grow along the spine to leaf 1. At each step, prior to adding the new leaf, suppose the spinal subtrees have masses $m_1,\ldots,m_k$. The new leaf attaches to one of the $j^{\text{th}}$ subtree with probability proportional to $m_j-\alpha$, or begins a new spinal subtree with probability proportional to $\alpha$ at any of the $k+1$ possible sites between adjacent spinal subtrees or at the top or bottom of the spine. This is the growth rule for an $(\alpha,\alpha)$ ordered Chinese restaurant process \cite{PW09}. As noted in \cite{PW09}, this means that after $\mu(B)$ growth steps, the sequence of subtree masses going down the spine is an $(\alpha,\alpha)$ regenerative composition structure, as in \cite[Section 8.4]{gp}. Therefore, the mass of the first such subtree is distributed according to the \emph{decrement matrix}, which we have quoted from \cite{gp} in \eqref{eq:returnalpha}. This proves the first assertion.
 
 By the same argument as in the proof of Lemma \ref{lem:spatialMarkov}, the spinal subtrees, including $U$, are conditionally independent of each other given the sequence of their leaf counts $(M_i)_{i\in [K]}$, each with law $q_{M_i,\alpha}$. This gives the claimed conditional law for $U$. Moreover, by definition of regenerative composition structures, the subtree masses down the spine from leaf $1$ in $V$ also form an $(\alpha,\alpha)$ regenerative composition structure. And again, by the argument in Lemma \ref{lem:spatialMarkov}, this sequence of masses is conditionally independent of $U$ given $M_1$. This is enough information to specify the conditional law of $\int(V)$, given $U$, as $\widetilde{q}_{n-M_1,\alpha}$.
\end{proof}

\section{Down-up chains on decorated trees}

\subsection{Definitions}\label{Sect23}
We now show that both the alpha chain and the uniform chain induce Markov chains on decorated $k$-trees with total mass $n$. The idea is very general. Suppose we have a time-homogeneous Markov chain $(\mathcal{X}(j))_{j\ge 0}$, with a finite state space $\bX$. Let $P$ denote the transition matrix for this Markov chain. Let $\bY$ be another finite space and let $g\colon\bX \rightarrow \bY$ be a surjective function. For each $y \in\bY$, suppose we have a probability function $\Lambda(y, \cdot)$ on $\bX$. Then, one can define a Markov operator $S$ acting on any function $h\colon\bY\rightarrow \rr$ by
\eq\label{eq:Mkvfn}
S\left( h\right)(y) := \sum_{x\in\bX} \Lambda\left( y,x \right) \pE\left[ h\circ g\left( \mathcal{X}(1) \right) \mid \mathcal{X}(0)=x \right].
\en
We will denote the corresponding transition matrix on $\bY$ by $Q=\Lambda P g$, where we abuse notation and write $g$ for the 
transition matrix $(\mathbf{1}_{(g(x) = y)})_{x\in\bX,y\in\bY}$ that corresponds to deterministic transitions from each $x\in\bX$ to $g(x)\in\bY$.
If $(\mathcal{X}(j))_{j\ge 0}$ has a stationary distribution $q$, a natural choice for $\Lambda(y, \cdot)$ is the conditional distribution 
$q(\,\cdot\,|g=y)$. In particular, this holds when $(\mathcal{X}(j))_{j\ge 0}$ is the alpha chain (or, the uniform chain), $q$ is 
$q_{n,\alpha}$ (or, the uniform distribution) on $\tshape{n}$ and $g= \rho^{\bullet(n)}_k$. In this section we show that the resulting 
Markov transition kernel $Q$ on $\mathbb{T}^{\bullet(n)}_{[k]}$ has an autonomous description that does not refer to the chains on 
$n$-trees. This is useful if we wish to send $n$ to infinity while keeping $k$ fixed. Throughout, we will consider the parameter 
$\alpha$ restricted to the open interval $(0,1)$.

Recall the Dirichlet-multinomial (${\rm DM}$) distributions. This is a family of distributions that takes two parameters, a positive integer $m$ and a vector $v$ in $(0, \infty)^d$, for some $d\ge 2$. Given such $m$ and $v$, this is a discrete probability distribution on the set of nonnegative integers $(j_1, \ldots, j_d)$ that add up to $m$ that can be described as follows. Pick $(p_1, \ldots, p_{d})$ from a Dirichlet distribution with parameters $v$. Let $(J_1, \ldots, J_{d})$ be conditionally distributed according to a multinomial distribution of counts when sampling with replacement $m$ balls of $d$ colors, where at each step the ball of color $i$ has a probability $p_i$ of getting sampled. Then the marginal distribution of $(J_1, \ldots, J_{d})$, when integrated over $(p_1, \ldots, p_{d})$, is ${\rm DM}$ with parameters $m$ and $v$. For our purpose we take $d=2k-1$, for some $1\le k < n$, and, for some $\alpha\in (0,1)$, the first $k$ coordinates of $v$ are equal to $1-\alpha$ and the rest are equal to $\alpha$. Denote this distribution by $\DirM^\alpha_{2k-1}(m)$. 
%
%
We also define $\DirM^\alpha_{2k-1}(0)$ to be the probability distribution that puts mass one on the zero-vector of dimension $2k-1$.



Just as in the case of $n$-trees, the induced Markov chains on decorated $k$-trees involve a down-move followed by an up-move. However, the down- and up-moves can be of different kinds depending on the decoration. The essential point is that an external edge in the 
$k$-tree is not at risk of being dropped until its mass is one and the mass of its parent is zero. We now define these special cases first. Fix $1\le k < n$ and $\alpha\in (0,1)$ throughout.

\begin{defn}\label{defn:droplabel}\rm{(Dropping a label)}
Let $\ft^\bullet=\left( \fs, (x_1, \ldots, x_k, (y_B)_{B\in \fs,\; \# B\ge 2}) \right) \in \bT^{\bullet(n)}_{[k]}$. Suppose label $i$ is such that $x_i=1$ and $y_{\sparent{i}}=0$. Then, by \emph{dropping label $i$ from $\ft^\bullet$} we mean performing the following sequentially.
\begin{enumerate}[(i)]
\item Drop leaf $i$ from $\fs$ and relabel leaves $i+1, \ldots, k$ by $i, \ldots, k-1$, respectively, to get a tree shape $\fs^{(i)}\in \tshape{k-1}$.
\item Drop $x_i$ and $y_{\sparent{i}}$ from the list of weights in $\ft^\bullet$ and relabel the remaining weights by the edges of $\fs^{(i)}$ to obtain an element of $\bT_{[k-1]}^{\bullet(n-1)}$.
\end{enumerate}
\end{defn}


\begin{defn}[Inserting a label]\label{defn:insert} Suppose $\ft^\bullet=\left( \fs,\left(x_1,\ldots,x_k,(y_B)_{B\in\fs,\; \# B\ge 2}\right) \right) \in \bT_{[k-1]}^{\bullet(n-1)}$. By \emph{inserting the label $k$ in $\ft^\bullet$} we mean performing the following actions on $\ft^\bullet$ sequentially to obtain a $\bT_{[k]}^{\bullet(n-1)}$-valued random variable. 
  \begin{enumerate}[(i)]
  \item Choose an edge $B$ with probability 
  \[
  \begin{dcases}
 \frac{x_i-1}{n-k},&\quad \text{if $B=\{i\}$ is external},\\
  \frac{y_B}{n-k},& \quad \text{if $B$ is internal}. 
  \end{dcases}
  \]
  \item If an external edge $\{i\}$ has been chosen, sample three integers $(j_1^*, j_2^*, j_3^*)$ using the kernel $\DirM^\alpha_3\left(x_i-2 \right)$. Update the tree shape $\fs$ to $\fs+(\{i\},k)$, by inserting a new leaf labeled $k$ on edge $\{i\}$. Split the corresponding $x_i$ into three random integers $\left( x_i, x_{k},y_{\{i,k\}}  \right)$ where
  \[
 x_i=j^*_1+1,\quad x_{k}=j^*_2+1,\quad  y_{\{ i,k\}}=j^*_3.
  \] 
  \item If an internal edge has been chosen, sample three integers $(j_1^*, j_2^*, j_3^*)$ using the kernel $\DirM^{1-\alpha}_3\left(y_B-1 \right)$. Update the tree shape $\fs$ to $\fs+(B,k)$, by inserting a new leaf labeled $k$ on edge $B$. Split the original $y_B$ in three random integers $\left( y_{B\cup\{k\}}, y_{B}, x_{k}  \right)$ where 
  \[
  y_{B\cup \{k\}}=j^*_1, \quad y_B=j^*_2,\quad x_{k}=j^*_3+1.
  \]
  \end{enumerate}
\end{defn}

\begin{defn}[Up-move on decorated trees]\label{defn:decupmove}$\;$\\  Let $\ft^\bullet=\left( \fs,\left(x_1,\ldots,x_k,(y_B)_{B\in\fs,\; \# B\ge 2}\right) \right) \in \bT_{[k]}^{\bullet(n-1)}$. 
Choose an edge $B$ with probability
\[
\begin{dcases}
\frac{x_i-\alpha}{n-1-\alpha},&\quad \text{if $B=\{i\}$ is external},\\
\frac{y_B+ \alpha}{n-1-\alpha},& \quad \text{if $B$ is internal},
\end{dcases}
\]
and increase the mass of the chosen edge by one to obtain an element of $\bT_{[k]}^{\bullet(n)}$.
\end{defn}

For the uniform decorated chain we combine variations of dropping a label (without relabeling leaves), up-move and inserting a label (using the edge chosen for the up-move also for the label insertion), in a single resampling step. 

\begin{defn}[Resampling a label]\label{defn:resample} Let $\ft^\bullet=\left( \fs, (x_1, \ldots, x_k, (y_B)_{B\in \fs,\; \# B\ge 2}) \right) \in \bT^{\bullet(n)}_{[k]}$. Suppose label $i$ is such that $x_i=1$ and $y_{\sparent{i}}=0$. Then, by \em resampling label $i$ \em from $\ft^\bullet$ we mean the following list of actions to get a $\bT_{[k]}^{\bullet(n)}$-valued random variable.
\begin{enumerate}[(i)]
\item Drop leaf $i$ from $\fs$ to get a tree shape $\fs^{(i)}\in \bT_{[k]\backslash \{i\}}$.
\item Drop $x_i$ and $y_{\sparent{i}}$ from the list of weights in $\ft^\bullet$.
\item Choose an edge $B$ of $\fs^{(i)}$ with probability 
\[
\begin{dcases}
\frac{x_j- 1/2}{n-3/2},&\quad \text{if $B=\{j\}$ is external},\\
\frac{y_B+ 1/2}{n-3/2},& \quad \text{if $B$ is internal}.
\end{dcases}
\]
\item If an external edge $B=\{j\}$ has been chosen, sample three integers $(j_1^*, j_2^*, j_3^*)$ using the kernel $\DirM^{1/2}_3\left(x_j-1 \right)$. Increase the current $x_j$ by one and split into three new random integers $\left( x_j, x_{i},y_{\{i,j\}}  \right)$ where
  \[
 x_j=j^*_1+1,\quad x_{i}=j^*_2+1,\quad  y_{\{ i,j\}}=j^*_3.
  \]
 Update the tree shape $\fs$ to $\fs^{(i)}+(\{j\},i)$, inserting a new leaf $i$ below leaf $j$. 
  \item If an internal edge $B$ has been chosen, sample three integers $(j_1^*, j_2^*, j_3^*)$ using the kernel $\DirM^{1/2}_3\left(y_B \right)$. Increase the original $y_B$ by one and split into three new random integers $\left( y_{B\cup\{i\}}, y_{B}, x_{i}  \right)$ where 
  \[
  y_{B\cup \{i\}}=j^*_1, \quad y_B=j^*_2,\quad x_{i}=j^*_3+1.
  \]
  Update the tree shape $\fs$ to $\fs^{(i)}+(B,i)$, inserting a new leaf $i$ below vertex $B$. 
\end{enumerate}
\end{defn}

We will now define a Markov chain on the space of decorated $k$-trees $\decotree{k}{n}$ of total edge mass $n$. 

\begin{defn}[Uniform decorated chain]\label{defn:simpdownup} Fix $n\ge 3$ and $k\in[n]$. We define a Markov chain $(T_k^{\bullet(n)}(j))_{j\ge 0}$ on the state space $\bT_{[k]}^{\bullet(n)}$ with the following transition rules. 
Given $T_k^{\bullet(n)}(j)=\ft^\bullet=\left( \fs, (x_1, \ldots, x_k, (y_B)_{B\in \fs,\; \# B\ge 2}) \right)\in \bT_{[k]}^{\bullet(n)}$, $j\ge 0$, randomly construct $T_k^{\bullet(n)}(j+1)$ by first making the following random selection:  
\begin{enumerate}[(i)]
\item select the internal edge $B$ for any $B\in\fs:\#B\ge 2$ with probability $y_B/n$, or
\item select the external edge to leaf $i$ for any $i\in[k]$ with probability $x_i/n$.
\end{enumerate} 
Then, perform the following random move in the case determined by this selection.
\begin{enumerate}
\item[(A)] If some internal $B$ or an external $i$ with $x_i\ge 2$ is chosen, reduce the chosen weight by one and perform an up-move as in Definition \ref{defn:decupmove} with $\alpha=1/2$.  
\item[(B)] If some external $i$ with $x_i=1$ is chosen and $y_{\sparent{i}}>0$, sample a random $m$ from the kernel $\delta\left( y_{\sparent{i}}: \cdot \right)$ in \eqref{eq:return}. Set $x_i=m$ and reduce $y_{\sparent{i}}$ by $m$. Perform an up-move as in Definition \ref{defn:decupmove}.
\item[(C)] If some external $i$ with $x_i=1$ is chosen and $y_{\sparent{i}}=0$, consider the $k$-tree shape $\fs$ and compare $i$ with the smallest leaf labels $a$ and $b$ in the first two subtrees on the ancestral path from leaf $i$ to the root (with the convention $b=0$ if the path has only one subtree). Let $\wi=\max\{i,a,b\}$. Swap $i$ and $\wi$, if different. Resample $\wi$ as in Definition \ref{defn:resample}.
\end{enumerate}
\end{defn}

The above chain can be generalized for all $\alpha \in (0,1)$ by using the alpha chain given in Definition \ref{defn:ntreeupdown}. 

\begin{defn}[Alpha decorated chain]\label{downmove} Fix $\alpha \in (0,1)$, $n\ge 3$, $k\in[n]$. We define $(T_k^{\bullet(n)}(j))_{j\ge 0}$. 
Given $T_k^{\bullet(n)}(j)=\left( \fs, (x_1, \ldots, x_k, (y_B)_{B\in \fs,\; \# B\ge 2}) \right)\in \bT_{[k]}^{\bullet(n)}$, $j\ge 0$, randomly construct $T_k^{\bullet(n)}(j+1)$ by first making the following random selection:  
\begin{enumerate}[(i)]
\item select the internal edge $B$ for any $B\in\fs:\#B\ge 2$ with probability $y_B/n$, or
\item select the external edge to leaf $i$ for any $i\in[k]$ with probability $x_i/n$.
\end{enumerate} 
Then, perform the following random move in the case determined by this selection.
\begin{enumerate}
\item[(A)] If some internal $B$ or an external $i$ with $x_i\ge 2$ is chosen, reduce the chosen weight by one and perform an up-move as in Definition \ref{defn:decupmove}.  
\item[(B)] If some external $i$ with $x_i=1$ is chosen and $y_{\sparent{i}}>0$, sample a random $m$ from the kernel $\delta_{\alpha}\left( y_{\sparent{i}}: \cdot \right)$ in \eqref{eq:return}. Set $x_i=m$ and reduce $y_{\sparent{i}}$ by $m$. Perform an up-move as in Definition \ref{defn:decupmove}.
\item[(C)] If some external $i$ with $x_i=1$ is chosen and $y_{\sparent{i}}=0$, consider the $k$-tree shape $\fs$ 
  and compare $i$ with the smallest leaf labels $a$ and $b$ in the first two subtrees on the ancestral path from leaf $i$ to the root 
  ($b=0$ if the path has only one subtree). Let $\wi=\max\{i,a,b\}$. Swap $i$ and $\wi$, if different. Drop $\wi$ from $\ft^\bullet$ as 
  in Definition \ref{defn:droplabel} to get an element in $\bT_{[k-1]}^{\bullet (n-1)}$. Insert label $k$ as in Definition 
  \ref{defn:insert} and then perform an up-move as in Definition \ref{defn:decupmove}. 
 \end{enumerate}
\end{defn}


\subsection{Representation of uniform and alpha decorated transition matrices}

Recall the discussion following \eqref{eq:Mkvfn}. 

\begin{prop}\label{prop:kupdown}
  Let $P$ denote the Markov transition kernel for the uniform chain of Definition \ref{defn:simptreeupdown} and let 
  $\Lambda(\ft^\bullet,\cdot)$ denote the conditional law $q_{n,1/2}(\,\cdot\,|\,\rho^{\bullet(n)}_k=\ft^\bullet)$. 
  Then, the Markov transition kernel for the uniform decorated chain of Definition \ref{defn:simpdownup} is 
  $\Lambda P \rho^{\bullet(n)}_k$. Similarly, let $P^{(\alpha)}$ denote the transition kernel of the alpha chain of Definition 
  \ref{defn:ntreeupdown}. Then the Markov transition kernel for the alpha decorated chain in Definition \ref{downmove} is 
  $\Lambda^{(\alpha)} P^{(\alpha)} \rho^{\bullet(n)}_k$, where $\Lambda^{(\alpha)}(\ft^{\bullet},\cdot)=q_{n,\alpha}(\,\cdot\,|\, \rho^{\bullet(n)}_k=\ft^\bullet)$.
\end{prop}

To prove this proposition we first describe the decorated $k$-trees as a process when projected from an alpha growth process $(T_n, n\ge 1)$. Let $k\ge 2$. Condition on $T_k=\fs$ and consider the conditional law of the decorated $k$-tree process 
$$
T_{k}^{\bullet(k+r)}:=\rho_{k}^{\bullet(k+r)}(T_{k+r}),\; r\ge 0.
$$
Clearly every $T_{k}^{\bullet(k+r)}$ has the same $k$-tree shape given by $\fs$ and nondecreasing edge weights. Let $x_j(r)$ and $y_B(r)$, for $j\in [k]$ and $B \in \fs$, $\#B\ge 2$, denote the masses assigned to edges of $\fs$ in $T_{k}^{\bullet(k+r)}$. 

\begin{lemma}\label{lem:lempolya} The joint distribution of the multi-dimensional process
$$
\left( x_1(r)-1, \ldots, x_k(r)-1, \; \left( y_B(r) \right)_{B \in \fs:\; \#B \ge 2}    \right),\; r\ge 0,
$$
is given by a generalized P\'{o}lya urn model $H(r)=(H_B(r), B\!\in\! \fs )$, $r\!\ge\! 0$, with weights 
$$
\begin{cases}
\alpha+H_B(\cdot),& \text{when $B$ is internal, and}\\
1-\alpha+ H_B(\cdot), & \text{when $B$ is external}.
\end{cases}
$$
The initial condition is $H_B(0)=0$ for all $B\in \fs$.\pagebreak[2]
\end{lemma}

In this P\'{o}lya urn, the edges $B\in\fs$ are the colors in the urn, and the model follows the standard update rule, i.e.\ selecting colors with probabilities proportional to weights and setting $H_{B}(r+1)=H_B(r)+1$ if color $B$ is selected at step $r+1$. See \cite[Proposition 18]{HMPW}, \cite[Proposition 14]{PW09}.

It follows from Lemma \ref{lem:lempolya} that $(T_{k}^{\bullet(k+r)})_{r\ge 0}$ is a Markov chain with a limiting Dirichlet distribution for the weights. The following corollaries are easy consequences of Lemma \ref{lem:lempolya}. 

\begin{cor}\label{cor:insertion_unlabel} The up-move in Definition \ref{defn:decupmove} describes the conditional distribution of $T_{k}^{\bullet(n)}$, given $T_{k}^{\bullet(n-1)}=\ft^\bullet$.
\end{cor}

\begin{cor}\label{cor2} Let $T_k \sim q_{k,\alpha}$. Given $T_k = \fs$, use the edges as an indexing set where the external edges are labeled by $1,2,\ldots,k$ corresponding to the leaf labels, and the internal edges are labeled (e.g. in lexicographical order) by $k+1, \ldots, 2k-1$. Now, generate a vector $\left( J_1, \ldots, J_{2k-1} \right)\sim \DirM^\alpha_{2k-1}(n-k)$. Assign $J_i$ to the edge labeled $i$ for each $i \in [2k-1]$. For $i\in [k]$ define $x_i = J_i+1$, and for $B\in \fs$, $\#B\ge 2$, set $y_B = J_l$, where $l\in [k+1,2k-1]$ is the label of edge $B$. 
Then $\left( \fs,  (x_1, \ldots, x_k, (y_B)_{B \in \fs:\; \# B\ge 2})  \right)$ has the same distribution as $\rho_{k}^{\bullet(n)}$ under $q_{n,\alpha}$.    
\end{cor}

\begin{proof} By definition, $T_k$ is distributed as $q_{k,\alpha}$. The ${\rm DM}^\alpha_{2k-1}$ distribution is due to the well-known exchangeability of the P\'olya urn model. If we let $r\rightarrow \infty$ in Lemma \ref{lem:lempolya}, the proportion of counts of each color has a limit $\mathbf{p}$ that is distributed according to the Dirichlet distribution with a parameter vector $v$ that is $\alpha$ in the first $k$ coordinates and $(1-\alpha)$ in the rest. By de Finetti's theorem, given a realization of this limiting proportion $\mathbf{p}$, the counts of balls is conditionally multinomial with probability vector $\mathbf{p}$. This, by definition, is $\DirM^{\alpha}_{2k-1}(n-k)$.
\end{proof}


We now vary $k$ in Lemma \ref{lem:lempolya}. Fix $n \in \NN$, and consider the stochastic process $\left( T^{\bullet(n)}_{k},\; k=1,2,\ldots, n  \right)$. It is not hard to see that this is also a Markov chain. To describe its transition probabilities, notice that the alpha growth process gives the following relation between $T^{\bullet(n)}_{k-1}$ and $T^{\bullet(n)}_{k}$. When we attach the $k$th leaf to $T_{k-1}$, it introduces two new edges on which the future leaves can get attached. If we imagine the count of leaves growing on the edges of $T_{k-1}$ as an urn model, the $k$th leaf chooses one of the colors in $T_{k-1}$ at random and splits it into three new colors. The transition kernel from $T^{\bullet(n)}_{k-1}$ to $T^{\bullet(n)}_{k}$ is given by the conditional probability of (i) which of the original colors got split and (ii) what the proportions of the new colors are after all $n$ leaves have been attached.

\begin{lemma}\label{ins} Let $T_{n-1}\sim q_{n-1,\alpha}$, set $T^{\bullet(n-1)}_k=\rho_k^{\bullet(n-1)}(T_{n-1})$ and
  $T^{\bullet(n-1)}_{k-1}=\rho_{k-1}^{\bullet(n-1)}(T_{n-1})$. Then inserting a label, as in Definition \ref{defn:insert}, describes the
  conditional distribution of $T^{\bullet(n-1)}_{k}$ given 
  $T^{\bullet(n-1)}_{k-1}=\left( \fs,\left(x_1,\ldots,x_k,(y_B)_{B\in\fs,\; \# B\ge 2}\right) \right) \in \bT_{[k-1]}^{\bullet(n-1)}$. 
\end{lemma}
\begin{proof} Consider the P\'{o}lya urn set-up from Lemma \ref{lem:lempolya} and Corollary \ref{cor2}. Let $r=n-k$.
Fix $B^*\in \fs$. Given $H_B(r)=j_B$, $B \in \fs$, the conditional probability that the $k$th leaf got attached to $B^*$ is $j_{B^*}/r$ by exchangeability of the urn scheme. Once the $k$th leaf gets attached, the three new colors evolve as a P\'{o}lya urn model themselves for time $r-1$. Let $B^*_1$, $B^*_2$, $B^*_3$ be the three new colors. Then $H_{B^*_1}(0)=H_{B^*_2}(0)=H_{B^*_3}(0)=0$. If $j^*_1, j^*_2, j^*_3$ form the state of the three colors at time $r-1$ (when $n-1$ leaves have been attached), then, we must have $j^*_1+j^*_2+j^*_3=j_{B^*}-1$. 

Suppose $B^*$ is external. Then the introduction of the new leaf creates two external edges (say $B_1^*$ and $B_2^*$) and one internal ($B_3^*$). Thus the distribution of 
$\left(  H_{B_1^*}(r-1), H_{B_2^*}(r-1), H_{B_3^*}(r-1) \right)$, given $H_{B^*}(r)=j_{B^*}=x_{B^*}-1$, is $\DirM$ with parameters $x_{B^*}-2$ and $(1-\alpha, 1-\alpha, \alpha)$. 
On the other hand, if $B^*$ is internal, the introduction of the new leaf creates two internal edges and one external. This gives us a ${\rm DM}$ distribution with parameters $y_{B^*}-1$ and $(\alpha, \alpha, 1-\alpha)$, where the first two coordinates stand for the internal edges and the last one for the external edge. 

Hence, the edge selection and splitting probabilities are as in Definition \ref{defn:insert}.
\end{proof}

\begin{lemma}\label{lem:revamp} Fix $i\in [k]$. Let $T_n\sim q_{n,1/2}$, $T_{n-1}=T_n\cap[n-1]$, $T^{\bullet(n)}_k=\rho_k^{\bullet(n)}(T_n)$ and 
  $T^{\bullet(n-1)}_{[k]\backslash \{i\}}:=\rho^{\bullet(n-1)}_{[k]\backslash \{i\}}(T_{n-1})$. Then performing steps (iii), (iv) and (v)
  of Definition \ref{defn:resample}, describes the conditional distribution of $T_k^{\bullet(n)}$ given $T^{\bullet(n-1)}_{[k]\backslash \{i\}}=\ft^\bullet\in\bT_{[k]\setminus\{i\}}^{\bullet(n-1)}$.
\end{lemma}

\begin{proof} By exchangeability of leaf labels in $q_{n,1/2}$, it is enough to consider the case when $i=k$. Thus, we compute the 
  conditional distribution of $T^{\bullet (n)}_{k}$, given $T^{\bullet(n-1)}_{k-1}$. The remainder of this proof is very similar to the 
  previous proof, and hence we give an outline of the argument and leave the details for the reader.  

  Let $\fs\in\bT_{k-1}$. We condition on $\{T_{k-1}=\fs\}$. Consider an up-move that adds the mass of a new leaf $n$ to
  $T^{\bullet(n-1)}_{k-1}$ at a randomly chosen edge $B\in \fs$ according to weights as in Definition \ref{defn:decupmove} for 
  $\alpha=1/2$. By Corollary \ref{cor:insertion_unlabel}, this is a step by the conditional distribution of $T^{\bullet(n)}_{k-1}$ 
  given $T^{\bullet(n-1)}_{k-1}$. By exchangeability, again, this leaf $n$ could be relabeled leaf $k$, and the total leaf mass of $T_n$ 
  can now be projected onto the tree spanned by leaves labeled by $[k]$. This will split the weight of edge $B$ according to a $\DirM$ 
  distribution as in the previous proof.
  Hence, the edge selection, mass insertion and splitting probabilities are as in Definition \ref{defn:resample} (iii)-(v).
\end{proof}

\begin{proof}[Proof of Proposition \ref{prop:kupdown}] We only prove the case of the alpha chain. The case of the uniform chain is 
  similar and is left to the reader. 

  Fix $\alpha\in (0,1)$ and 
  $\ft^\bullet = \left( \fs, (x_1, x_2, \ldots, x_k, (y_B)_{B\in \fs,\; \# B\ge 2}) \right) \in \decotree{k}{n}$. 
  Proposition \ref{prop:kupdown} claims the equality of two distributions on $\decotree{k}{n}$, one obtained by a transition of the 
  alpha decorated chain as defined in Definition \ref{downmove}, the other by transitions via the alpha chain. Specifically, let us 
  consider random variables that exhibit the latter transition, with the alpha chain transition further split into the down-move and the 
  up-move, as follows:
  \begin{itemize}
    \item an initial $n$-tree $T_n$, which is distributed according to $\Lambda(\ft^\bullet, \cdot)$;
    \item a leaf label $I$, uniformly distributed over $[n]$ and independent of $T_n$, that is initially selected for removal in the 
      down-move; then $T_{n-1}^\prime \in \tshape{n-1}$ is obtained from $T_n$ and $I$ by dropping $\widetilde{I}$ and 
      relabeling as in Definition \ref{defn:ntreeupdown};
    \item an edge $S$ of $T_{n-1}^\prime$ distributed as in the up-move of Definition \ref{defn:ntreeupdown}; then the tree
      $T_n^\prime=T_{n-1}^\prime+(S,n)\in \tshape{n}$ realises the up-move and we want to study the distribution of
      $\rho_k^{\bullet(n)}(T_n^\prime)$. 
  \end{itemize}
  To compare this distribution with the alpha decorated chain, we further define on the same probability space the common tree shape 
  $T_k^\prime=T_{n-1}^\prime\cap[k]$ of $\rho_k^{\bullet(n-1)}(T_{n-1}^\prime)$ and $\rho_k^{\bullet(n)}(T_n^\prime)$, and the edges 
  $\beta\in\fs$ and $\beta^\prime\in T_k^\prime$ to which $I$ and $S$ project, respectively. Let us introduce $\beta$ and 
  $\beta^\prime$ more formally. 
  
  Let $\pi$ be the projection map (see above Definition \ref{defn:collapse}) that takes leaves of $T_n$ to edges in $\fs$. Let 
  $\beta=\pi(I)$ denote the edge of $\fs$ on which $I$ projects. The conditional probability $\pP\left( \beta=B \mid T_n \right)$, for 
  any edge $B\in \fs$, is clearly $\mu(B)/n$. Note that this probability is independent of $T_n$, thus $\beta$ and $T_n$ are independent
  and $\pP(\beta=B)=\mu(B)/n$ has the same distribution as the edge selected in the transition of the alpha decorated chain. It
  therefore suffices to identify the conditional distribution of $\rho_k^{\bullet(n)}(T_n^\prime)$ given $\beta=B$ with the distribution 
  of the random move specified in (A), (B) or (C) of Definition \ref{downmove}, as the case may be for each $B\in\fs$.
  
  Let $\pi^\prime$ be the projection from $T_{n-1}^\prime$ to $\rho_k^{\bullet(n-1)}(T_{n-1}^\prime)$. By construction of 
  $T_n^\prime=T_{n-1}^\prime+(S,n)$, there is a unique edge $\beta^\prime\in T_k^\prime$ whose weight changes when passing from
  $\rho_k^{\bullet(n-1)}(T_{n-1}^\prime)$ to $\rho_k^{\bullet(n)}(T_{n}^\prime)$, and this weight increases by one. 
  Then, for any edge $B^\prime\in T_k^\prime$, the conditional probability $\pP\left( \beta'=B' \mid T_n,I\right)$ is proportional to 
  $\mu^\prime(B^\prime)-\alpha$ for external edges $B'$, or $\mu'(B')+\alpha$ for internal edges, where $\mu^\prime(B^\prime)$ is the weight of the edge $B^\prime$ in $\rho_k^{\bullet(n-1)}(T_{n-1}^\prime)$. In 
  particular, it is a function only of $\rho_k^{\bullet(n-1)}(T_{n-1}^\prime)$, as is required as the last step in all three cases
  (A), (B) and (C) of Definition \ref{downmove}.

\smallskip

\noindent\textbf{Case (A).} Consider an edge $B\in \fs$ such that either (i) $B$ is internal with $\mu(B)>0$ or (ii) $B$ is external and $\mu(B) >1$. 
  Let us study the effect of the alpha chain down-move on the decorated tree $\rho_k^{\bullet(n-1)}(T_{n-1}^\prime)$ when conditioning
  on $\{\beta=B\}$.  

  (i) First suppose $B$ is internal. If $I=i$, obviously $i>k$. Every leaf label $j$ such that $\pi(j)=B$ must also be larger than $k$. 
  In particular, $\wi \ge i > k$. Hence, once we swap and drop $\wi$, we reduce $y_B$ by one without changing any other weights on 
  $\ft^\bullet$. 

  (ii) Now suppose $B$ is an external edge $\{i\}$ such that $x_i>1$. Consider the set of leaves $j$ in $T_n$ such that $\pi(j)=\{i\}$. 
  Since $x_i >1$, there are at least two of them, including $i$ itself. All such $j$, $j\neq i$, must be at least $k+1$. Among them, 
  during the $n$-tree down-move, either some $j\neq i$ is removed, which is similar to the paragraph above, or the leaf $i$ itself is 
  chosen for removal. For the latter, let $\tilde{\imath}$ be the leaf that gets swapped with $i$ and is dropped in the down-move. 
  Consider the first two subtrees on the ancestral path from $i$ to the root in $T_n$.  Since $x_i > 1$, the leaf labels in at least one 
  of the subtrees is contained in $\pi^{-1}(\{i\})$, and therefore $\tilde{\imath} > k$. In fact, $\pi(\tilde{\imath})$ must be $\{i\}$. 
  To see this, note that, either both subtrees are contained in $\pi^{-1}(\{i\})$, in which case the conclusion is obvious, or, the 
  first subtree is in $\pi^{-1}(\{i\})$, while the second subtree is not. In the latter case, the second subtree must have a leaf $\le k$
  (consider a leaf label in the sibling of $\{i\}$ in $\fs$) while in the first subtree all leaves must be labeled $\ge k+1$. Thus 
  $\tilde{\imath}$, being the maximum of the two minimum labels in two subtrees, must belong to the first subtree and hence 
  $\tilde{\imath}\ge k+1$ and $\pi(\tilde{\imath})=\{i\}$. Therefore, after swapping and removing $\tilde{\imath}$ and relabeling higher 
  labeled leaves simply reduces $x_i$ by one and has no other effects on $\ft^\bullet$. 

  Thus, in both (i) and (ii), $\rho_k^{\bullet(n-1)}(T_{n-1}^\prime)$ has the same tree shape as $\ft^\bullet$ and all the same weights 
  except 
  on edge $B$, whose weight 
  is reduced by one, as needed for the down-move in Definition \ref{downmove}. 
  The up-move probabilities have already been discussed.

\smallskip

\nin\tbf{Case (B).} Suppose $B=\{i\}\in \fs$ for some $i\in [k]$ such that $x_i=1$ and  \vspace{-0.2cm} $y_{\sparent{i}}>0$. We 
  condition on $\{\beta=B\}$.  
  Let $\{i\}'$ and $\sparent{i}$ be the labels of the sibling edge and the parent edge of $\{i\}$ in $\fs$, the tree 
  shape in $\ft^\bullet$. Consider the spinal decomposition (recall Section \ref{sec:prelim}) from leaf $\{i\}$ to the root in $T_n$. Let
 \[
  U:=\bigcup_{B\in \fs:\; B\subseteq \{i\}'}\pi^{-1}\left( B\right)\quad \text{and}\quad 
  V:=\pi^{-1}\left( \sparent{i}  \right).
 \]
 Since $x_i = 1$, $U$ is the set of labels in the first subtree on the ancestral line from $\{i\}$ to the root, while the labels in the second subtree comprise a non-empty subset $W\subseteq V$. Since $U$ contains some leaf label at most $k$, whereas every leaf in $V$ (and thus $W$) has label at least $k+1$, we conclude that $\tilde\imath$ comes from $W$.
 
 In fact, the above $U$ and $W$ are edges in $T_n$: the sibling and uncle of $\{i\}$. Let $U'$ and $W'$ denote the resulting label sets after replacing $\tilde\imath$ with $i$ in $W$ and reducing all labels greater than $\tilde\imath$ by one. The step of dropping $\tilde\imath$ in the down-move makes $U'$ and $W'$ siblings in $T'_{n-1}$. Set $M_i := \#W \le \#V = y_{\sparent{i}}$ in $\ft^{\bullet}$. 
 All weights and labels are the same in $\rho_k^{\bullet(n-1)}(T'_{n-1})$ as in $\ft^{\bullet}$, with two exceptions: weight $x_i=1$ is replaced by $M_i$ and weight $y_{\sparent{i}}$ is reduced by $M_i$. Therefore, the law of $\rho_k^{\bullet(n-1)}(T_{n-1}^\prime)$ under $\pP( \,\cdot\,|\,\beta=B)$ is specified by the conditional law of $M_i$. Since $T_n$ has law $q_{n,\alpha}(\,\cdot\,|\,\rho_k^{\bullet(n)}=\ft^\bullet)$ and is independent of $\beta$, it follows from Lemmas \ref{lem:spatialMarkov} and \ref{lem:topbeadkernel} that the law of $M_i$ is given by the stochastic kernel $\delta_{\alpha}\big(y_{\sparent{i}}:\cdot\big)$ of \eqref{eq:returnalpha}.
  

  Again, the up-move probabilities have been discussed, which completes this case.

\smallskip

\nin\tbf{Case (C).}  Suppose $B=\{i\}$ with $x_i=1$ and $y_{\sparent{i}}=0$
  in $\ft^\bullet$. We condition on $\{\beta=B\}$.  
  Since $x_i=1$ and $y_{\sparent{i}}=0$, each of the first two subtrees on the ancestral line of $\{i\}$ in $T_n$ 
  must have a leaf label at most $k$. Hence, the $n$-tree down-move drops $\wi\le k$ and shifts labels $(\wi+1, \ldots, n)$ down to
  by one. This gives us $T_{n-1}^\prime$. However, unlike the previous cases, the tree shape of $\rho_k^{\bullet(n-1)}(T_{n-1}^\prime)$ is not 
  deterministic under $\pP(\,\cdot\,|\,\beta=B)$. This is because leaf $k+1$ was relabeled as $k$, so the weight on $\{k\}$ in $\rho^{\bullet(n-1)}_{k}(T'_{n-1})$ corresponds to part of some other, larger weight in $\ft^{\bullet}$, according to where $k+1$ was placed when generating $T_n$. 

  Now note that relabeling $(k+1, \ldots, n)$ does not affect the weights on the decorated trees. Thus 
  $\fv^\bullet=\rho^{\bullet(n-1)}_{k-1}\left( T_{n-1}^\prime \right)$ is deterministic under $\pP(\,\cdot\,|\,\beta=B)$. To be consistent 
  with Definition \ref{downmove}, it is enough to argue (due to Lemma \ref{ins}) that the distribution of $T_{n-1}^\prime$ under 
  $\pP(\,\cdot\,|\,\beta=B)$ is $q_{n-1,\alpha}(\,\cdot\,|\,\rho^{\bullet(n-1)}_{k-1}=\fv^\bullet)$. However, this is very similar to the proof of 
  Lemma \ref{lem:downinv} since $T_n$ is distributed as $q_{n,\alpha}(\,\cdot\,|\,\rho^{\bullet(n)}_k=\ft^\bullet)$. In fact, if one repeats 
  the proof of Lemma \ref{lem:downinv} for decorated $k$-trees, instead of $n$-trees, we reach our conclusion. The up-move is as usual. 
  This completes the proof. 
\end{proof}

Before we end this section, let us extend the result in Corollary \ref{cor:resampdist}. 

\begin{cor}\label{cor:firstresmpdist}
Consider an alpha decorated chain (Definition \ref{downmove}) running in stationarity. Let $\widetilde{I}$ be the first leaf label in $[k]$ that gets dropped in case (C). Then, the distribution of $\widetilde{I}$ is given by \vspace{-0.3cm}
\eq\label{eq:distwiktree}
\pP\left(  \widetilde{I}=\wi \right)=\begin{dcases}
\frac{1}{k(k-1-\alpha)}, & \text{if $\wi=2$.}\\
\frac{2\wi-2-\alpha}{k(k-1-\alpha)}, & \text{if $3\le \wi \le k$}.
\end{dcases}
\en
For $\alpha=1/2$ this is also the distribution of the first resampled leaf label in the uniform decorated chain (Definition \ref{defn:simpdownup}). 
\end{cor}

\begin{proof}
Let $(T(r))_{r\ge 0}$ be an alpha chain on $\tshape{n}$ running in stationarity. As before, let $(T^{\bullet(n)}_k(r))_{r\ge 0}$ denote 
the projected decorated chain. Let $\tau$ denote the first time $r$ such that a label in $[k]$ is dropped. Note that $\tau$ is a 
stopping time with respect to the natural filtration of $(T^{\bullet(n)}_k(r))_{r\ge 0}$. 

Let $\fs$ denote the $k$-tree shape of $T^{\bullet(n)}_k(0)$. By our assumption of stationarity, $\fs$ is distributed according to 
$q_{k,\alpha}$ over $\tshape{k}$. By definition of $\tau$, each of the decorated trees $T^{\bullet(n)}_k(r)$, $0\le r\le \tau-1$, has 
the same trees shape $\fs$. At time $\tau$, let $I\in [k]$ denote the random leaf label that is selected for removal. That is, $I=i$ is 
the label satisfying case (C) in Definition \ref{defn:simpdownup}. The statement of the corollary will follow by the exact same proof as 
for Corollary \ref{cor:resampdist} if we show that $I$ is uniformly distributed over $[k]$, independent of $\fs$. 

However, this follows from symmetry of the conditional law $T^{\bullet(n)}_k(0)$, given the tree shape $\fs$, as described in Corollary 
\ref{cor2}. It shows the following exchangeability property: suppose we relabel the leaves of $\fs$ by a permutation $\sigma$ of $[k]$. For simplicity assume that $\sigma$ transposes leaf labels $1$ and $2$. Switch the labels of the leaves $1$ and $2$ in $T^{\bullet(n)}_k(0)$ to get a new decorated $k$-tree. The tree-shape of this new tree is identical to $\fs$, except for the switching of leaf labels $1$ and $2$. However, the conditional law of the both the decorated $k$-trees, given their tree shapes, is identical and given by Corollary \ref{cor2}. This follows from the exchangeability of the first $k$ coordinates of $\DirM^{\alpha}_{2k-1}(n-k)$, given the last $(k-1)$ coordinates. That is, if we generate a vector from $\DirM^{\alpha}_{2k-1}(n-k)$, switch coordinates $1$ and $2$, while leaving the rest of the coordinates unchanged, the distribution remains unchanged.

Now, this initial coupling of the two decorated trees extends to a natural coupling of the alpha decorated chains where we perform the same moves on the unlabeled trees but keep the swapped leaf labels $1$ and $2$. Thus, we have a measure-preserving bijection between the sample paths in the events $\{I=1\}$ and $\{I=2\}$, given $\fs$. Hence the probability of the two events, given $\fs$, must be the same. Replacing $2$ by any other label $\wi\in [k]$ shows the conditional uniform distribution of $I$. Hence, $I$ and $\fs$ are independent and the proof is complete.
\end{proof}

\section{Consistency in stationarity for the uniform down-up chains} For the rest of the paper we consider the uniform chain of Definition \ref{defn:simptreeupdown} and the uniform decorated chain of Definition \ref{defn:simpdownup}. Our main result is
Theorem \ref{thm projectedchain} of the Introduction. With the preparations done so far, we restate it here, as follows. 

\begin{thm}[Consistency in stationarity]\label{thm:runningalongbottom} 
  Fix $n\ge 2$. Let $(T(j))_{j\ge 0}$ be a uniform chain running in stationarity. For $1\le k \le n-1$, let 
  \[
  T^{\bullet (n)}_k(j) = \rho_{k}^{\bullet(n)}(T(j)), \quad j\ge 0,
  \]
  denote the corresponding sequence of decorated $k$-trees obtained by projections. Then each $(T_k^{\bullet(n)}(j))_{j\ge 0}$, is a 
  uniform decorated chain as described in Definition \ref{defn:simpdownup}, also running in stationarity.   
\end{thm}

Recall the concept of collapsed trees from Definition \ref{defn:collapse}. Define the process
\[
T^{\star(n)}_k(j):=\rho_k^{\star(n)}\left( T(j)  \right), \quad j\ge 0.
\]
Thus, on the same filtered probability space that supports $(T(j))_{j\ge 0}$ we have three stochastic processes 
$T(\cdot):=(T(j))_{j\ge 0}$, $T^{\star(n)}_k(\cdot):=(T^{\star(n)}_k(j))_{j\ge 0}$ and 
$T^{\bullet(n)}_k(\cdot):=(T_k^{\bullet(n)}(j))_{j\ge 0}$. We will show that, when the first is running in stationarity, all three 
processes are Markov chains by themselves, also running in stationarity.  

Recall the set-up described at the beginning of Section \ref{Sect23}. Suppose we have a Markov chain 
$\mathcal{X}(\cdot):=(\mathcal{X}(j))_{j\ge 0}$, with a finite state space $\bX$, a transition kernel $K(x, \cdot)$, and an initial probability 
distribution $\mu$. Let $\bY$ be another finite state space and let $g\colon\bX \rightarrow \bY$ be a given surjective function. We ask 
the question: when is $\mathcal{Y}(j):=g\left( \mathcal{X}(j)  \right)$, $j\ge 0$, a Markov chain?

There have been a number sufficient conditions proposed for this problem. We will require two of them: (i) the Kemeny--Snell 
criterion and (ii) the intertwining criterion due to Rogers and Pitman. We 
describe both these conditions below. The proof of Theorem \ref{thm:runningalongbottom} will consist of 
two key steps.  The first step is to show that the pair $\mcal{X}(\cdot)=T^{\star(n)}_k(\cdot)$, $\mcal{Y}(\cdot):=T^{\bullet(n)}_k(\cdot)$ 
satisfies the Kemeny--Snell criterion, where $T^{\star(n)}_k(\cdot)$ is a Markov chain with a transition kernel induced from $T(\cdot)$ as described in \eqref{eq:Mkvfn}. The next step is to show that the pair $\mcal{X}(\cdot):=T(\cdot)$, $\mcal{Y}(\cdot):=T^{\star(n)}_k(\cdot)$ satisfies the Rogers--Pitman intertwining criterion. Combining the two gives us Theorem  \ref{thm:runningalongbottom}. We are forced to take this two-step approach since it can be seen from explicit calculation and examples that neither intertwining nor Kemeny--Snell hold between the chains $\mcal{X}(\cdot):=T(\cdot)$ and $\mcal{Y}(\cdot):=T^{\bullet(n)}_k(\cdot)$.

Let us now describe the Kemeny--Snell \cite{KS60} criterion. Consider the Markov transition kernel from $\bX$ to $\bY$ given by
\[
Q_g(x,\{y\}) := K\left(x, g^{-1}(\{y\})\right), \quad y \in \bY. 
\] 
Suppose that the kernel $Q_g(x,\cdot)$ is a function of $g(x)$, i.e., for any $y\in \bY$,
\[
Q_g(x_1, y)= Q_g(x_2, y),\quad \text{if}\; g(x_1)=g(x_2) \quad (\text{Kemeny--Snell criterion}).
\]
For any pair of $y_1, y_2 \in \bY$, denote the common value of $Q_g\left(  x_1,  y_2 \right)$, for any $x_1 \in g^{-1}(\{y_1\})$, by $Q\left( y_1, y_2 \right)$. The following result is in \cite[Theorem 6.3.2]{KS60}.

\begin{lemma}\label{lem:KSlemma}If the Kemeny--Snell criterion holds, then $\mathcal{Y}(\cdot)$ is a Markov chain with transition kernel $Q(\cdot, \cdot)$ and an initial distribution given by the push-forward of the measure $\mu$ by the function $g$. In particular, $\mathcal{Y}(\cdot)$ is running in stationarity if $\mathcal{X}(\cdot)$ is running in stationarity. 
\end{lemma}

Consider again the set-up at the beginning of Section \ref{Sect23}. Take the Markov chain $\mcal{X}$ to be the down-up chain $T(\cdot)$ and the function $g=\rho^{\star(n)}_k$. Then, by \eqref{eq:Mkvfn}, there is a natural transition probability kernel on the set of collapsed trees $\bT^{\star(n)}_{[k]}$. We now take $\bX=\bT^{\star(n)}_{[k]}$ and $K(\cdot, \cdot)$ to be this induced probability kernel. Redefine $g$ to denote the projection map that takes a collapsed tree $\ft^{\star}$ and replaces each $\pi^{-1}(B)$ by the mass $\mu(B)=\#\pi^{-1}(B)$ to obtain the decorated tree with masses $\ft^\bullet$. 

\begin{prop}\label{lem:kscollapsed}
If we take $\mcal{X}(j)=T^{\star(n)}_k(j)$, $j\ge 0$, with transition kernel $K$ and $\mcal{Y}(j)=T_k^{\bullet(n)}(j)$, then the Kemeny--Snell criterion holds.  
\end{prop}

\begin{proof}
Suppose we are given 
$\ft^\bullet(0), \ft^\bullet(1) \in \bT^{\bullet(n)}_k$ and $T^{\star(n)}_k(0)=\ft^\star(0)$ for some 
$\ft^\star(0)\in g^{-1}(\{\ft^\bullet(0)\})$ aiming to study $K(\ft^\star(0),g^{-1}(\{\ft^\bullet(1)\}))$. 
The induced kernel $K$ on $\bT^{\star(n)}_k$ can be described in the following way. Generate $T_n(0)$ from the 
conditional distribution $q_{n,1/2}(\,\cdot\,|\,\rho^{\star(n)}_k= \ft^\star(0))$ specified in Lemma \ref{lem:spatialMarkov}. Let 
$T_n(1)$ be one step in the uniform chain, starting at $T_n(0)$, and let $T^{\star(n)}_k(1)= \rho^{\star(n)}_k(T_n(1))$.  

Consider the unranked internal structures defined above Definition \ref{defn:internal_structure}. Specifically,  
let $\fs(0)$ and $T_k(1)$ denote the $k$-tree shapes at times zero and one, respectively. Then the internal structures
$\left(V_B(0), B\in\fs(0)\right)$ corresponding to $T_n(0)$ and $\left(V_B(1), B \in T_k(1)\right)$ corresponding to $T_n(1)$ are 
closely related. Indeed, we see by conditioning as in the different cases in the proof of Proposition \ref{prop:kupdown}, that $\pP(T^{\star(n)}_k(1)\in g^{-1}\left(\{\ft^\bullet(1)\}\right)|T_n(0))$ 
depends only on the ranked internal structures $\left( \int\left(V_B(0)  \right), \; B\in \fs(0)  \right)$ and not on the labels in 
$V_B(0)$, $B\in \fs(0)$. Specifically, only the splitting via the regenerative composition in Case (B) and via resampling in Case (C) depends on $T_n(0)$.   
By Lemma \ref{lem:spatialMarkov}, the internal structures are jointly independent of the labels in
$V_B(0)$, $B \in \fs(0)$. Hence, 
$K(\ft^\star(0),g^{-1}(\{\ft^\bullet(1)\}))=\pP(T^{\star(n)}_k(1)\in g^{-1}\left(\{\ft^\bullet(1)\}\right))$ does not depend on the labels
in $V_B(0)$, $B\in\fs(0)$, i.e.\ does not depend on $\ft^\star(0)\in g^{-1}(\{\ft^\bullet(0)\})$. This independence property is the 
Kemeny--Snell criterion.  
\end{proof}

Note that the Kemeny--Snell criterion fails between $\mathcal{X}(\cdot)=T(\cdot)$ and $\mathcal{Y}(\cdot)=T^{\bullet(n)}_k(\cdot)$ for 
instance because, for each edge $B\in\fs(0)$, the internal structure ${\rm int}(V_B)$ determines the size of the first spinal subtree, 
and hence influences the transition probabilities $K(\ft(0),g^{-1}(\{\ft^\bullet(1)\}))$.


\begin{figure}[t]
\centerline{
\xymatrix@=3em{
T_k^{\star(n)}(i) \ar[d]^{\Lambda} \ar[r]^{Q} & T_k^{\star(n)}(i+1) \ar[d]^{\Lambda} \\
T(i) \ar[r]^{P} & T(i+1)
}
}
\caption{Commutative diagram of intertwining of Markov chains on collapsed trees with that on $n$-trees.}
\label{fig:intertwin}
\end{figure}

Now we recall the concept of intertwining from Rogers and Pitman \cite{RogersPitman}. However, since we work in discrete time, unlike 
the continuous time setting in that paper, some of their statements get simplified. Once more consider the set-up described at the 
beginning of Section \ref{Sect23}. The Markov chain $\mcal{X}$ has transition kernel $P$ and a stationary distribution $q$. Let 
$\mcal{Y}(j)=g\left( \mcal{X}(j) \right)$. For $y\in \bY$, set $\Lambda(y, \cdot)=q(\,\cdot\,|\,g=y)$. 
Let $Q$ be a transition kernel on $\bY$ given by $\Lambda P g$, which corresponds to the operator \eqref{eq:Mkvfn}. The following is a 
simplification of \cite[Theorem 2]{RogersPitman} adapted to our case. \pagebreak

\begin{lemma}\label{lem:rogerspitman} Given the Markov chain $(\mcal{X}(j))_{j\ge 0}$, suppose 
\begin{enumerate}[(a)]
\item the distribution of $\mcal{X}(0)$, given $\mcal{Y}(0)=g(\mcal{X}(0))=y$, is $\Lambda(y, \cdot)$.
\item And, the triplet $(\Lambda, P, Q)$ satisfies the equality 
\eq\label{eq:intertwincrit}
\Lambda P= Q \Lambda,\qquad \text{(Intertwining criterion)}
\en
as an equality of two stochastic kernels. 
\end{enumerate}
Then $\mcal{Y}(\cdot)$ is a Markov chain with transition kernel $Q$. Moreover, the conditional distribution of any $\mcal{X}(j+1)$, 
given $\mcal{Y}(0), \ldots, \mcal{Y}(j)$, is $\Lambda\left(  \mcal{Y}(j), \cdot \right)$
\end{lemma}

\begin{proof} The function $\Phi$ in \cite{RogersPitman} is our function $g$. Hence condition (a) in \cite[Theorem 2]{RogersPitman} that 
$\Lambda g=I$, the identity kernel, obviously follows from our definition of $\Lambda$ as a conditional distribution, given $g$. 

To verify condition (b) in \cite[Theorem 2]{RogersPitman}, note that our time is discrete. Hence $P^j$, for $j\ge 0$, is the transition 
semigroup (the semigroup is denoted by $P_t$, $t\ge 0$ in \cite{RogersPitman}). Let $Q_j=\Lambda P^j g$, for $j\ge 0$. We claim that 
$Q_j=Q^j$, where the latter is $Q$ raised to the power $j$. This is true for $j=0$ by the previous paragraph, and for $j=1$ by 
\eqref{eq:intertwincrit}. Now, take $j=2$. From usual matrix multiplication, it follows that
\[
Q_2= \Lambda P^2 g= \left( \Lambda P \right) P g =\left( Q\Lambda\right) Pg = Q\left( \Lambda P\right) g = Q^2 \Lambda g= Q^2. 
\]
An induction over $j$ establishes our claim. Condition (b) in \cite{RogersPitman} requires $\Lambda P^j=Q_j \Lambda$. As $Q_j=Q^j$, we need to show $\Lambda P^j= Q^j \Lambda$. But this is obvious from \eqref{eq:intertwincrit}.
\end{proof}

\begin{prop}\label{lem:intertwin}
The Intertwining criterion \eqref{eq:intertwincrit} holds for the pair $\mcal{Y}(\cdot)=T_k^{\star(n)}(\cdot)$ and $\mcal{X}(\cdot)=T(\cdot)$ with the kernel $\Lambda$ given by the conditional distribution under $q_{n,1/2}$. 
\end{prop}

Intuitively the intertwining criterion refers to the commutation of the diagram in Figure \ref{fig:intertwin}. Let $P$ and $Q$ be the Markov transition kernels for the uniform chain on $n$-trees and for collapsed trees, respectively. Let 
$\Lambda(\ft^\star, \ft)=q_{n,1/2}(\ft\,|\,\rho^{\star(n)}_k=\ft^\star)$. The function $g$ is $\rho^{\star(n)}_k$. Suppose we are given 
the realization of $T^{\star(n)}_k(j)$, for some $j\ge 0$. Intertwining means that there are two equivalent ways to sample  
$T(j+1)$.
\begin{enumerate}[(i)]
\item Generate $T(j)$ from $\Lambda\left( T^{\star(n)}_k(j), \cdot\right)$ and run one step of the uniform $n$-tree chain to get $T(j+1)$.
\item Run one step of the Markov chain with transition $Q$, starting from $T^{\star(n)}_k(j)$, to get $T^{\star(n)}_k(j+1)$. Then, 
  generate $T(j+1)$ from $\Lambda\left( T^{\star(n)}_k(j+1),\cdot\right)$. 
\end{enumerate}

Before arguing Proposition \ref{lem:intertwin}, let us complete the proof of Theorem \ref{thm:runningalongbottom}. 

\begin{figure}[t]
\centerline{
\xymatrix@=3em{
T_k^{\bullet(n)}(i) \ar[r]^{R} & T_k^{\bullet(n)}(i+1)  \\
T_k^{\star(n)}(i) \ar[u] \ar[r]^{ Q } & T_k^{\star(n)}(i+1) \ar[u]\\
T(i)  \ar[u]  \ar[r]^{P} & T(i+1) \ar[u]
}
} 
\caption{Combining Kemeny--Snell with intertwining. $R$ is the transition kernel for the down-up chain on decorated $k$-trees. The up-arrows represent deterministic projection maps.}
\label{fig:intertwin2}
\end{figure}


\begin{proof}[Proof of Theorem \ref{thm:runningalongbottom}]
The proof follows from the commutativity of the diagram in Figure \ref{fig:intertwin2}. In fact, one can make a stronger statement. Suppose a realization of $T_k^{\bullet(n)}(0)$ is given. Generate $T_k^{\star(n)}(0)$, given $T_k^{\bullet(n)}(0)$, from the conditional distribution under $q_{n,1/2}$. This amounts to throwing, uniformly at random, balls labeled $k+1, \ldots, n$ to boxes labeled $[2k-1]$, one box for each edge of the tree shape of $T_k^{\star(n)}(0)$, such that the number of balls in each box is consistent with the corresponding edge mass in $T_k^{\bullet(n)}(0)$. Also generate $T(0)$ according to the kernel $\Lambda\left( T_k^{\star(n)}(0), \cdot \right)$. Now run the uniform $n$-tree chain starting from $T(0)$ to get a sequence $(T(j))_{j\ge 0}$. Set  
$T_k^{\star(n)}(j):= \rho^{\star(n)}_k\left( T(j) \right)$ and $T_k^{\bullet(n)}(j):= \rho^{\bullet(n)}_k \left( T(j) \right)$, $j\ge 0$.
By Lemma \ref{lem:rogerspitman} and Proposition \ref{lem:intertwin}, $T_k^{\star(n)}(\cdot)$ is a Markov chain with transition kernel $Q$
and initial state $T_k^{\star(n)}(0)$. Hence, by Lemma \ref{lem:kscollapsed}, $T_k^{\bullet(n)}(\cdot)$ is a Markov chain on decorated 
$k$-trees. When $T_k^{\bullet(n)}(0)$ is random and distributed according to the stationary distribution (which is the push-forward of 
the uniform distribution $q_{n,1/2}$ by the map $\rho^{\bullet(n)}_k$), we get the desired result. 
\end{proof}

To prove Proposition \ref{lem:intertwin} we need the following lemma.

\begin{lemma}\label{lem:treeremove2} 
  Let $T_m\sim q_{m,1/2}$ for some $m \ge 2$. Then, given $T_2^{\star(m)}:=\rho_2^{\star(m)}(T_m)=(T_m \cap [2], (L_B)_{B\in T_m\cap [2]})$, the internal 
  structures $\int(V_{\{1\}})$, $\int(V_{\{2\}})$ and $\int(V_{\{1,2\}})$ of $T_2^{\star(m)}$ are independent and uniformly distributed on the 
  set of binary trees with numbers of leaves given by $ \#L_{\{1\}}, \#L_{\{2\}}$ and $\#L_{\{1,2\}}+1$, respectively. 
\end{lemma}

\begin{proof}
  The proof is a direct application of Lemma \ref{lem:spatialMarkov}.
\end{proof}

\begin{proof}[Proof of Proposition \ref{lem:intertwin}]
By Remark (ii) in \cite[page 575]{RogersPitman}, the Intertwining criterion \eqref{eq:intertwincrit} is equivalent to showing that the conditional distribution of $\mcal{X}(1)$, given $\mcal{Y}(0)=y_0$ and $\mcal{Y}(1)=y_1$, is $\Lambda\left( y_1, \cdot \right)$. We verify this reformulated condition in our case. 

Fix $\ft^\star=\left( \fs, \left(L_B\right)_{B\in \fs}  \right)\in\bT_{[k]}^{\star(n)}$ throughout this proof. Consider the following 
uniform down- and up-move random variables, as in the proof of Proposition \ref{prop:kupdown}: 
\begin{itemize}\item an initial $n$-tree $T_n$, which is distributed according to $\Lambda(\ft^\star,\cdot)$;
  \item a leaf label $I$, uniformly distributed over $[n]$ and independent of $T_n$; then the swapped label $\widetilde{I}$ and the
    tree $T_{n-1}^\prime$ after the down-move (leaving the other labels unchanged as in 
    Definition \ref{defn:simptreeupdown}) are determined, as is the random edge $\beta$ of $\fs$ to which $I$ projects;
  \item an edge $S$ of $T_{n-1}^\prime$ distributed as in the up-move of Definition \ref{defn:simptreeupdown}; this determines
    $T_n^\prime=T_{n-1}^\prime+(S,\widetilde{I})$ after the up-move, as well as the edge $\beta^\prime$ of 
    $T_n^\prime\cap([k]\setminus\{\widetilde{I}\})$ to which $S$ projects (either to insert $\widetilde{I}$ into the collapsed
    tree after the down-move, if $\widetilde{I}>k$, or to split when $\widetilde{I}$ is resampled, if $\widetilde{I}\le k$).
\end{itemize}
We claim that the conditional distribution of $T_n^\prime$ given $\rho^{\star(n)}_k(T_n^\prime)$ is  
$\Lambda\left(\rho^{\star(n)}_k(T_n^\prime), \cdot  \right)$. We write 
$\rho^{\star(n)}_k(T_n^\prime)=(T_k^\prime,(L_B^\prime)_{B\in T_k^\prime})$, where 
$T_k^\prime=T_{n}^\prime\cap[k]$.

By the spatial Markov property Lemma \ref{lem:spatialMarkov}, the above claim amounts to showing that, given 
$\rho^{\star(n)}_k(T_n^\prime)=\fv^\star$, the internal structures of $\fv^\star$ in $T_n^\prime$ are conditionally 
independent trees that are conditionally uniformly distributed with sizes given by the label sets of $\fv^\star$. 

We will prove this claim on a case-by-case basis that mirrors the proof of Proposition \ref{prop:kupdown} showing that, conditioned on 
each case, the claim holds. 
\smallskip


\noindent\textbf{Case (A).} Consider any $i,\wi\in L_B$ for an edge $B\in \fs$ such that either (i) $B$ is internal or (ii) $B$ is 
external and $\#L_B>1$. Also consider any $B^\prime\in\fs$. We will condition on the event that $I=i$, $\widetilde{I}=\wi$ and 
$\beta^\prime=B^\prime$. Note that on this event the collapsed tree $\rho^{\star(n)}_k(T_n^\prime)$ after the down-move and up-move is 
determined as a deterministic function of $\ft^\star$, $\wi$ and $B^\prime$.

Consider a uniform growth process up to some step $m\ge 1$. Suppose we remove leaf $m$ and the edge below it, it follows from the uniform growth rules  
that the remaining subtree with leaves labeled by $[m-1]$ is uniformly distributed. By exchangeability of leaf labels, removing any 
other leaf $\wi\in[m-1]$, the internal structure of the remaining tree is uniformly distributed. Similarly, if we insert a new label, 
the new tree with the additional leaf is uniformly distributed with the given set of labels. 

In this case, as in the proof of Proposition \ref{prop:kupdown}, the dropped label $\wi$ is at least $k+1$. There is no 
resampling: the edge $B$ loses label $\wi$, the label set of the edge $B^\prime$ gains the label. Hence, the tree shape 
remains $\fs$ after the down-move and also after the up-move. By Lemma \ref{lem:spatialMarkov}, the internal structures of $\ft^\star$ 
in $T_n$ are mutually independent uniform binary trees under $\pP$, and since $I$ is independent of $T_n$, this also holds under 
$\pP(\,\cdot\,|\,I=i)$. Consider the internal structure $\int(V_B)$ of $\ft^\star$ in $T_n$ supported by edge
$B$.  Applying Lemma \ref{lem:downinv} to $\int(V_B)$ and further conditioning on $\widetilde{I}=\wi$, the 
internal structures after the down-move are independent under $\pP(\,\cdot\,|\,I=i,\widetilde{I}=\wi)$. By the previous paragraph, the 
same holds for the internal structures after the up-move, under $\pP(\,\cdot\,|\,I=i,\widetilde{I}=\wi,\beta^\prime=B^\prime)$, as 
required.

\smallskip

\nin\tbf{Case (B).} Consider any $i\in[k]$ with $x_i=1$ and $y_{\sparent{i}}>0$ and 
  $\wi\in L_{\sparent{i}}$. Also consider any $B^\prime\in\fs$. We will condition on the event that $I=i$, 
  $\widetilde{I}=\wi$ and $\beta^\prime=B^\prime$. Note that on this event the collapsed tree $\rho^{\star(n)}_k(T_n^\prime)$ is 
  not usually a deterministic function of $\ft^\star$, $i$, $\wi$ and $B^\prime$, since the down-move splits a new label set for 
  $\{i\}$ off the label set $L_{\sparent{i}}$. The tree shape is unaffected by this.




Let $B_1=\sparent{i}$ and consider the associated unranked internal structure $V_{B_1}$ of $\ft^\star$ in $T_n$. The 
swapped label $\wi$ is in $V_{B_1}$ as for the corresponding case in the proof of Proposition \ref{prop:kupdown}. Since $B_1$ is an 
internal edge of $\fs$, $V_{B_1}$ has a leaf labeled $1$ that corresponds to $B_1$ (see Figure \ref{fig:internal_structure}). 
Consider the first subtree $U$ on the spinal decomposition from leaf $1$ to the root in $V_{B_1}$. Then, during the down-move, leaf $i$ 
in $T_n$ swaps with the smallest label $\wi$ in this subtree to give us the new unranked internal structure of edge $\{i\}\in\fs$ in
$T_{n-1}^\prime$. Remove this subtree $U$ from $V_{B_1}$, but leave leaf $1$. The remaining subtree (say, $V$) in $V_{B_1}$ is precisely 
the new unranked internal structure of edge $B_1\in \fs$ in $T_{n-1}^\prime$. 
We note that their label sets, denoted by $L_{B_1}^{\prime\prime}$ and $L_{\{i\}}^{\prime\prime}$, determine 
$\rho^{\star(n-1)}_k(T_{n-1}^\prime)$ and are recorded in $\rho^{\star(n-1)}_k(T_{n-1}^\prime)$. The rest of the internal structures is 
not affected. Applying Lemma \ref{lem:topbeadkernel} to 
$\int(V_{B_1})$, we obtain that the internal structures after the down-move are independent under 
$\pP(\,\cdot\,|\,I=i,\widetilde{I}=\wi,L_{B_1}^{\prime\prime},L_{\{i\}}^{\prime\prime})$. The up-move is similar to the previous case. 

\smallskip

\noindent\textbf{Case (C).} Consider any $i\in[k]$ with $x_i=1$ and $y_{\sparent{i}}=0$. On $\{I=i\}$, this determines 
$\widetilde{I}=\wi\in[k]$ and the $(k-1)$-tree shape $T_n^\prime\cap([k]\setminus\{\widetilde{I}\})=\widetilde{\fv}$ after the down-move. 
Also consider $B^\prime\in\widetilde{\fv}$. As usual, condition on the event that  $I=i$, $\widetilde{I}=\wi$ and 
$\beta^\prime=B^\prime$.

In this case some label in $[k]$ is dropped and resampled. Hence the fact that the internal structures after the down-move are mutually 
independent uniform binary trees under $\pP(\,\cdot\,|\,I=i,\widetilde{I}=i)$ is obvious. We focus on the resampling step.  

When the label $\wi$ resamples it affects the collapsed tree after the down-move in the following way. The label $\wi$ is added to the 
label set of the edge $\beta'=B'$ increasing its weight by one, and the label set is then split according to a three-color P\'olya urn, 
as follows, into three sets that we denote by $L_{B^\prime}^{\prime}$, $L_{\{\wi\}}^\prime$ and $L_{B^\prime\cup\{\wi\}}^\prime$. If 
$B^\prime=\{j\}$ is external, initially the two lowest labels $\wi,j\in[k]$ have the first two colors. If $B^\prime$ is 
internal, the lowest label $\wi\in[k]$ is initially allocated the last color. In both cases, the other labels of $B^\prime$ are added 
in increasing order according to the standard update rules with initial weights $(1/2,1/2,1/2)$. This is so, because the internal 
structure supported by $B^\prime$ after the down-move is uniform and stays uniform after insertion of $\wi$, so that this split can be 
read from the proof of Corollary \ref{cor2}. In particular, this achieves the split of the edge weight of edge $B'$ according to the 
Dirichlet multinomial kernel given in Definition \ref{defn:resample}. This split determines $\rho^{\star(n)}_k(T_n^\prime)$ and is 
recorded in $\rho^{\star(n)}_k(T_n^\prime)$. Resampling further inserts the label corresponding to $\wi$ into $\int(V_{B^\prime})$ as in
the up-move of Definition \ref{defn:simptreeupdown} as either the smallest (external edge with $\wi<j$) or second smallest (internal 
edge or external with $\wi>j$) label. This results in a uniform tree, by uniform growth and exchangeability. By Lemma 
\ref{lem:treeremove2}, the three new internal structures are independent. The rest of the internal structures being unaffected, we 
conclude that the internal structures of $\rho^{\star(n)}_k(T_n^\prime)$ in $T_n^\prime$ are independent under
$\pP(\,\cdot\,|\,I=i,\widetilde{I}=\wi,\beta^\prime=B^\prime,L_{B^\prime}^{\prime},L_{\{\wi\}}^\prime,L_{B^\prime\cup\{\wi\}}^\prime)$.

\smallskip

Since the joint conditional distributions of internal structures given the events that appeared in the three (exhaustive) cases, only 
depend on the collapsed tree $\rho^{\star(n)}_k(T_n^\prime)$, this joint distribution is also the conditional distribution given just 
the collapsed tree. This completes the proof. 
\end{proof}

\appendix

\section{Trees decorated with strings of beads}

Although a decorated $k$-tree $\rho^{\bullet(n)}_{k}(\ft)$ contains information on the leaf masses of collapsed subtrees of $\ft$, it does not contain much information on lengths in these subtrees. 
For an illustration of this point, see the left-most tree in Figure \ref{fig:internal_structure}. The three red edges in $\ft$ contract to a single edge in $\rho^{\bullet(n)}_{k}(\ft)$. This number, $3$, is lost in the projection. To construct the Aldous diffusion as discussed in Section \ref{sec:intro:AD}, we must retain more of this length information. But if we simply enrich our state space by decorating edges with lengths, in addition to weights, then the resulting projected processes would be non-Markovian. To retain the Markov property while incorporating length information, we decorate each edge by a composition, which we call a \textit{(discrete) string of beads}.

For $m\in\bN$, denote by $\cC_m=\bigcup_{\ell\ge 0}\{(y_1,\ldots,y_\ell)\in\bN^\ell\colon y_1+\cdots+y_\ell=m\}$ the set of compositions of $m$, with the convention that $\cC_0$ is a singleton set containing the empty composition, denoted by `$()$'. We think of compositions as sequences of bead sizes, strung along a length of string. Continuum strings of beads were introduced to study weighted continuum random trees \cite{PW09,RembWinkString}. In that setting, a string of beads is a purely atomic measure on an interval. Here, the composition $(y_1,\ldots,y_\ell)$ could be thought of as describing a measure $y_1\delta_1+\cdots+y_\ell\delta_\ell$ on an interval $[0,\ell+1]$.

For $\fs\in\bT_{[k]}$, recall the set of edge weight functions $\cN_m^\fs$ from Definition \ref{defn:edgemass}. The set of \emph{$k$-trees decorated with strings of beads} is given by 
\[
 \bT_{[k]}^{\circ(n)} = \bigcup_{\fs\in\bT_{[k]}} \{\fs\} \times\left(  \bigcup_{f\in\cN_{n}^\fs}\{(f(\{1\}),\ldots,f(\{k\})\}\times\prod_{B\in\fs\colon\#B\ge 2}\cC_{f(B)}\right).
\]
We interpret an element $\ft^\circ = (\fs, x_1,\ldots,x_k,(y_1^B,\ldots,y_{\ell_B}^B)_{B\in\fs\colon\#B\ge 2})\in\bT_{[k]}^{\circ (n)}$ as a $k$-tree $\fs\in\bT_{[k]}$ with leaf masses $x_1,\ldots,x_k$, each at least one, and with each internal edge $B$ decorated by a (possibly empty) string of beads $(y_1^B,\ldots,y_{\ell_B}^B)$, with bead masses listed in order of increasing distance from the root. 
For example, in the right panel of Figure \ref{fig:k_tree_string}, the edges have decorations
\begin{gather*}
 x_1 = 1, \quad 
 x_2 = 1, \quad 
 x_3 = 1, \quad 
 x_4 = 3, \quad 
 x_5 = 2,\\
 \left(y^{\{1,4\}}_\bullet\right) = (1), \quad 
 \left(y^{\{2,5\}}_\bullet\right) = (), \quad 
 \left(y^{\{2,3,5\}}_\bullet\right) = (1,2), \quad 
 \left(y^{[5]}_\bullet\right) = (2).
\end{gather*}
Now, $\ell_B+1$ is the \emph{length of edge $B$} in $\ft^\circ$, and $y_1^B+\cdots+y_{\ell_B}^B$ is its \emph{mass}.

\begin{figure}\centering
 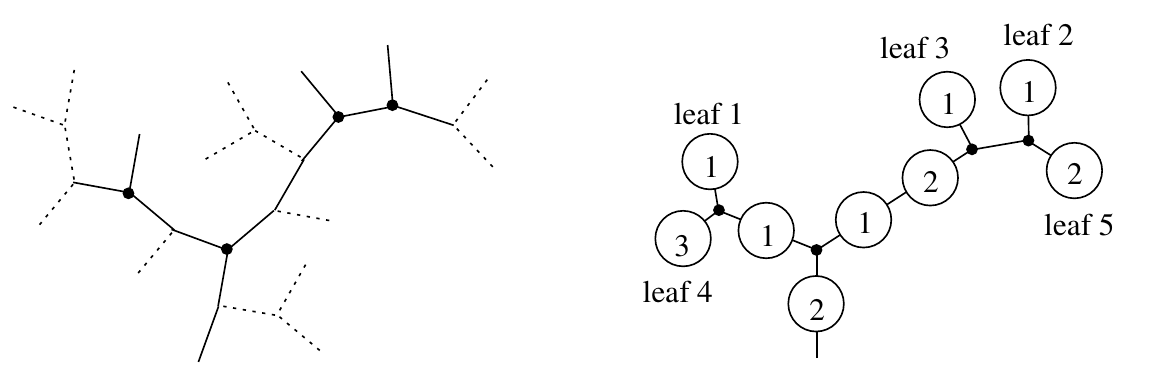
 \caption{$\bT_{[14]}\ni\ft\mapsto\rho^{\circ(14)}_{5}(\ft)\in\bT^{\circ(14)}_{5}$.\label{fig:k_tree_string}}
\end{figure}

Analogously to Section \ref{sec:decoration}, we define a projection $\rho^{\circ(n)}_k\colon\bT_{[n]}\rightarrow\bT_{[k]}^{\circ(n)}$. Fix $\ft\in\bT_{[n]}$. We define
\begin{equation*}
 \fr = \big\{B\in\ft\colon \#(B\cap [k])\ge 2\big\} \cup \left\{\bigcup\nolimits_{B\in\ft\colon B\cap[k] = \{j\}}B\ \middle|\ j\in [k]\right\}.
\end{equation*}
We view $\fr$ as a subgraph of $\ft$ formed by pruning away fringe subtrees that contain no labels from $[k]$, and if there is a path of degree-2 branch points terminating in a leaf, contracting that down to a single edge. See the left panel of Figure \ref{fig:k_tree_string}, in which the subgraph $\fr$ is shown in solid lines, with $\ft\setminus\fr$ shown in dashed lines. 

Denote by $\mu$ the measure on $\ft$ that assigns to each of the leaves $\{1\},\ldots,\{n\}$ mass 1. Consider the projection map $\pi_1\colon\ft\rightarrow\fr$ given by $B\mapsto \bigcap_{C\in\fr\colon C\supseteq B}C$. 
Denote by $\widetilde\mu$ the $\pi_1$-pushforward of $\mu$. Note that, for $B\in\fr$, $\widetilde\mu\{B\} = 0$ if and only if $B$ is a degree-3 branch point in $\fr$. Indeed, if $B$ is a leaf in $\fr$ then its $\pi_1$-pre-image is a subtree containing one leaf with low label ($\le k$), and possibly more high-labeled leaves. If $B$ is a degree-2 branch point in $\fr$ then its $\pi_1$-pre-image is a fringe subtree that branches off of $\fr$ at $B$, which therefore contains at least one high-labeled leaf. But if $B$ has degree 3 in $\fr$ then its $\pi_1$-pre-image is only $\{B\}$.

Now define $\fs := \ft\cap [k] = \fr\cap [k]$, and let $\pi_2\colon \fr\to\fs$ denote the map $B\mapsto B\cap [k]$. This map bijects the leaves and degree-3 branch points of $\fr$ with the vertices of $\fs$, but the branch points of $\fs$ may also have degree-2 vertices of $\fr$ in their $\pi_2$-pre-image. For $j\in [k]$ we define 
$x_j$ to be the $\widetilde\mu$-mass of the leaf of $\fr$ that maps to $\{j\}$ via $\pi_2$. 
For $B\in\fs$ with $\#B\ge 2$, let $\ell_B := \#\pi_2^{-1}(B) - 1$. For the purpose of the following, denote the elements of $\pi_2^{-1}(B)$ by $B_1,\ldots,B_{\ell_B+1}$, in order of increasing distance from the root, so that $B_1,\ldots,B_{\ell_B}$ all have degree 2 and $B_{\ell_B+1}$ has degree 3. We define $\big(y^B_j\big){}_{j\in [\ell_B]} := (\widetilde\mu(B_j))_{j\in [\ell_B]}$. 
Finally, 
\[
 \rho^{\circ(n)}_k(\ft) := \left((x_j)_{j\in[k]},\big(y_1^B,\ldots,y_{\ell_B}^B\big)_{B\in\fs\colon \#B\ge2}\right).
\]

\begin{thm}\label{thm projectedchainstrings}
Let $(T(j))_{j\geq 0}$ be a uniform chain running in stationarity.  Then, for $k\leq n$, $(\rho^{\circ(n)}_k(T(j)))_{j\geq 0}$ is a Markov chain running in stationarity. 
\end{thm}
The proof of this extension of Theorem \ref{thm projectedchain} is very similar to that already given, so we omit it. We note this because we will give a completely different proof of the analogous result in the continuum, and the discrete proofs presented in this paper will shed more light on the intuition behind those results. 

We also note that the discrete strings of beads considered here could also be taken as representations of finite \textit{interval partitions}. In the continuum, however, there is a crucial advantage working with interval partitions rather than strings of beads (purely atomic measures). This is the approach taken in \cite{Paper1}, which the reader can consult for a discussion of interval partitions in this setting.

As a final remark, one may wonder whether the down-up chain in Definition \ref{defn:simptreeupdown} is the only modification of label dynamics for the Aldous chain that would make Theorem \ref{thm projectedchain} possible. We found our down-up chain as a natural way to project the stationary Aldous chain down to a stationary evolving $2$-tree with leaf masses and a string of beads. This idea goes back to \cite{Pal2011}. By working our way upwards, from $2$ up to $n$ leaves, while keeping the consistency in stationarity, we arrived at Definition \ref{defn:simptreeupdown}. The proof presented here is reversed from the order in which it was discovered, beginning with label dynamics for $n$ leaves and then projecting down, and only considering edge masses rather than strings of beads for simplicity.

\bibliographystyle{abbrv}
\bibliography{ad}

\def\cprime{$'$}
\begin{thebibliography}{10}

\bibitem{Ald-91}
D.~Aldous.
\newblock The continuum random tree. {I}.
\newblock {\em Ann. Probab.}, 19(1):1--28, 1991.

\bibitem{AldousDiffusionProblem}
D.~Aldous.
\newblock Wright-{F}isher diffusions with negative mutation rate!
\newblock \url{http://www.stat.berkeley.edu/~aldous/Research/OP/fw.html}, 1999.

\bibitem{Aldous00}
D.~J. Aldous.
\newblock Mixing time for a {M}arkov chain on cladograms.
\newblock {\em Combin. Probab. Comput.}, 9(3):191--204, 2000.

\bibitem{LVB}
D.~Barker.
\newblock {\em LVB 1.0: reconstructing evolution with parsimony and simulated
  annealing}.
\newblock Daniel Barker, 1997.

\bibitem{nCRP}
D.~M. Blei and M.~I. Jordan.
\newblock The nested {C}hinese restaurant process and {B}ayesian nonparametric
  inference of topic hierarchies.
\newblock {\em J. ACM}, 57(2), 2010.

\bibitem{BoroFerr14}
A.~Borodin and P.~L. Ferrari.
\newblock Anisotropic growth of random surfaces in 2+ 1 dimensions.
\newblock {\em Communications in Mathematical Physics}, 325(2):603--684, 2014.

\bibitem{BoroOlsh09}
A.~Borodin and G.~Olshanski.
\newblock Infinite-dimensional diffusions as limits of random walks on
  partitions.
\newblock {\em Probab. Theory Related Fields}, 144(1-2):281--318, 2009.

\bibitem{DiacFill90}
P.~Diaconis and J.~A. Fill.
\newblock Strong stationary times via a new form of duality.
\newblock {\em Ann. Probab.}, 18(4):1483--1522, 1990.

\bibitem{EW}
S.~N. Evans and A.~Winter.
\newblock Subtree prune and regraft: a reversible real tree-valued {M}arkov
  process.
\newblock {\em Ann. Probab.}, 34(3):918--961, 2006.

\bibitem{felsenstein}
J.~Felsenstein.
\newblock {\em Inferring Phylogenies}.
\newblock Sinauer, 2003.

\bibitem{Fill92}
J.~A. Fill.
\newblock Strong stationary duality for continuous-time {M}arkov chains. {I}.
  {T}heory.
\newblock {\em J. Theoret. Probab.}, 5(1):45--70, 1992.

\bibitem{For-05}
D.~J. Ford.
\newblock {\em Probabilities on cladograms: introduction to the alpha model}.
\newblock Ph. D. thesis, Stanford University, 241 p., article version also
  available at arXiv:math.PR/0511246, 2006.

\bibitem{Paper1}
N.~Forman, S.~Pal, D.~Rizzolo, and M.~Winkel.
\newblock Diffusions on a space of interval partitions with
  {P}oisson-{D}irichlet stationary distributions.
\newblock arXiv:1609.06706, 2016.

\bibitem{Fulman05}
J.~Fulman.
\newblock Stein's method and {P}lancherel measure of the symmetric group.
\newblock {\em Trans. Amer. Math. Soc.}, 357(2):555--570, 2005.

\bibitem{Fulman09b}
J.~Fulman.
\newblock Commutation relations and {M}arkov chains.
\newblock {\em Probab. Theory Related Fields}, 144(1-2):99--136, 2009.

\bibitem{Fulman09}
J.~Fulman.
\newblock Mixing time for a random walk on rooted trees.
\newblock {\em Electron. J. Combin.}, 16(1):Research Paper 139, 13, 2009.

\bibitem{gp}
A.~Gnedin and J.~Pitman.
\newblock Regenerative composition structures.
\newblock {\em Ann. Probab.}, 33(2):445--479, 2005.

\bibitem{HMPW}
B.~Haas, G.~Miermont, J.~Pitman, and M.~Winkel.
\newblock Continuum tree asymptotics of discrete fragmentations and
  applications to phylogenetic models.
\newblock {\em Ann. Probab.}, 36(5):1790--1837, 2008.

\bibitem{KS60}
J.~Kemeny and J.~Snell.
\newblock {\em {Finite Markov chains}}.
\newblock University series in undergraduate mathematics. Van Nostrand, 1960.

\bibitem{Kerov96}
S.~Kerov.
\newblock The boundary of {Y}oung lattice and random {Y}oung tableaux.
\newblock In {\em Formal power series and algebraic combinatorics ({N}ew
  {B}runswick, {NJ}, 1994)}, volume~24 of {\em DIMACS Ser. Discrete Math.
  Theoret. Comput. Sci.}, pages 133--158. Amer. Math. Soc., Providence, RI,
  1996.

\bibitem{nHDP}
J.~Paisley, C.~Wang, D.~M. Blei, and M.~I. Jordan.
\newblock Nested hierarchical {D}irichlet processes.
\newblock {\em IEEE Transactions on pattern analysis and machine intelligence},
  37(2), 2015.

\bibitem{Pal2011}
S.~Pal.
\newblock {On the Aldous diffusion on Continuum Trees. I}.
\newblock {\em arXiv:1104.4186 [math.PR]}, 2011.

\bibitem{Pal13}
S.~Pal.
\newblock Wright-{F}isher diffusion with negative mutation rates.
\newblock {\em Ann. Probab.}, 41(2):503--526, 2013.

\bibitem{Pang17}
C.~Y.~A. Pang.
\newblock Lumpings of algebraic {M}arkov chains arise from subquotients.
\newblock arXiv: 1508.01570v2 [math.CO].

\bibitem{Pet2009}
L.~A. Petrov.
\newblock A two-parameter family of infinite-dimensional diffusions on the
  {K}ingman simplex.
\newblock {\em Funktsional. Anal. i Prilozhen.}, 43(4):45--66, 2009.

\bibitem{PW09}
J.~Pitman and M.~Winkel.
\newblock Regenerative tree growth: binary self-similar continuum random trees
  and {P}oisson-{D}irichlet compositions.
\newblock {\em Ann. Probab.}, 37(5):1999--2042, 2009.

\bibitem{RembWinkString}
F.~Rembart and M.~Winkel.
\newblock Recursive construction of continuum random trees.
\newblock To appear in {\it Ann.\ Probab.}. Preprint at arXiv:1607.05323
  [math.PR], 2016.

\bibitem{remy}
J.-L. R{\'e}my.
\newblock Un proc\'ed\'e it\'eratif de d\'enombrement d'arbres binaires et son
  application \`a leur g\'en\'eration al\'eatoire.
\newblock {\em RAIRO Inform. Th\'eor.}, 19(2):179--195, 1985.

\bibitem{RogersPitman}
L.~C.~G. Rogers and J.~W. Pitman.
\newblock Markov functions.
\newblock {\em The Annals of Probability}, 9(4):573--582, 1981.

\bibitem{mrbayes}
F.~Ronquist, M.~Teslenko, P.~Van Der~Mark, D.~L. Ayres, A.~Darling,
  S.~H{\"o}hna, B.~Larget, L.~Liu, M.~A. Suchard, and J.~P. Huelsenbeck.
\newblock {MrBayes 3.2: efficient Bayesian phylogenetic inference and model
  choice across a large model space}.
\newblock {\em Systematic biology}, 61(3):539--542, 2012.

\bibitem{Schweinsberg02}
J.~Schweinsberg.
\newblock An {$O(n^2)$} bound for the relaxation time of a {M}arkov chain on
  cladograms.
\newblock {\em Random Structures Algorithms}, 20(1):59--70, 2002.

\end{thebibliography}

\end{document}